\documentclass[12pt]{article}

\usepackage{amsmath}
\usepackage{amssymb}
\usepackage{microtype}
\usepackage{amsthm}
\usepackage[margin=30mm]{geometry}
\usepackage[usenames]{color}
\usepackage{graphicx}
\usepackage{tikz}
\usetikzlibrary{arrows,decorations.pathmorphing,decorations.footprints,fadings,calc,trees,mindmap,shadows,decorations.text,patterns,positioning,shapes,matrix,fit,through,decorations.pathreplacing}

\newtheorem{thm}{Theorem}[section]
\newtheorem{defn}[thm]{Definition}
\newtheorem{lemma}[thm]{Lemma}

\newtheorem{cor}[thm]{Corollary}

\newtheorem{claim}[thm]{Claim}

\begin{document}

\title{Combinatorially interpreting generalized Stirling numbers}

\author{John Engbers\thanks{john.engbers@marquette.edu; Department of Mathematics, Statistics and Computer Science,
Marquette University, Milwaukee WI 53201. Research supported by the Simons Foundation.} ~~~David Galvin\thanks{dgalvin1@nd.edu; Department of Mathematics,
University of Notre Dame, Notre Dame IN 46556, + 1 574 6814181 (corresponding author). Research supported by NSA grant H98230-13-1-0248, and by the Simons Foundation.} ~~~Justin Hilyard\thanks{jhilyard@nd.edu; Department of Mathematics,
University of Notre Dame, Notre Dame IN 46556. Research supported by NSA grant H98230-13-1-0248.}}

\date{\today}

\maketitle

\begin{abstract}
The Stirling numbers of the second kind ${n \brace k}$ (counting the number of partitions of a set of size $n$ into $k$ non-empty classes) satisfy the relation
$$
\left(xD\right)^nf(x) = \sum_{k \geq 0} {n \brace k} x^k D^kf(x)
$$
where $f$ is an arbitrary function and $D$ is differentiation with respect to $x$. More generally, for every word $w$ in alphabet $\{x,D\}$ the identity
$$
wf(x) = x^{\left(\#\left(\mbox{$x$'s in $w$}\right)-\#\left(\mbox{$D$'s in $w$}\right)\right)}\sum_{k \geq 0} S_w(k) x^k D^kf(x)
$$
defines a sequence $\left(S_w(k)\right)_k$ of {\em Stirling numbers (of the second kind)} of $w$. Explicit expressions for, and identities satisfied by, the $S_w(k)$ have been obtained by numerous authors, and combinatorial interpretations have been presented.

Here we provide a new combinatorial interpretation that, unlike previous ones, retains the spirit of the familiar interpretation of ${n \brace k}$ as a count of partitions. Specifically, we associate to each $w$ a quasi-threshold graph $G_w$, and we show that $S_w(k)$ enumerates partitions of the vertex set of $G_w$ into classes that do not span an edge of $G_w$. We use our interpretation to re-derive a known explicit expression for $S_w(k)$, and in the case $w=\left(x^sD^s\right)^n$ to find a new summation formula linking $S_w(k)$ to ordinary Stirling numbers. We also explore a natural $q$-analog of our interpretation.

In the case $w=\left(x^rD\right)^n$ it is known that $S_w(k)$ counts increasing, $n$-vertex, $k$-component $r$-ary forests.  Motivated by our combinatorial interpretation we exhibit bijections between increasing $r$-ary forests and certain classes of restricted partitions.

\end{abstract}

\section{Introduction}

The Stirling number of the second kind, ${n\brace k}$, counts the number of ways of partitioning a set of $n$ elements into $k$ non-empty classes. It satisfies the recurrence
\begin{equation} \label{eq-Stirling-recurrence}
{n \brace k} = \left\{
\begin{array}{ll}
{n-1 \brace k-1} + k{n-1 \brace k} & \mbox{if both $n>0$ and $k>0$, and} \\
{\bf 1}_{\{n=k\}} & \mbox{if $nk=0$.}
\end{array}
\right.
\end{equation}
The numbers ${n\brace k}$ satisfy numerous algebraic identities; indeed, it was through the identity
$$
x^n = \sum_{k \geq 0} {n \brace k} x^{\underline{k}}
$$
for $n \geq 0$, where $x^{\underline{k}}$ is the $k^{th}$ falling power $x(x-1)\ldots (x-(k-1))$, that James Stirling originally introduced the numbers, in his 1730 {\em Methodus differentialis} \cite{Stirling}. Central to the present paper is the following identity, probably first observed by Scherk in his 1823 thesis \cite{Scherk} (see \cite{BlasiakFlajolet} for details), which arises when one repeatedly applies the operator $xD$ to an infinitely differentiable function $f(x)$ (where $D$ is differentiation with respect to $x$). For all $n \geq 0$ we have
\begin{equation} \label{eq-Stirling-as-op}
(xD)^nf(x) = \sum_{k \geq 0} {n\brace k} x^k D^kf(x)
\end{equation}
(where here and throughout we interpret $D^0$ as the identity). One way to verify (\ref{eq-Stirling-as-op}) is to prove by induction on $n$ that $(xD)^n f(x)$  takes the form $\sum_{k \geq 0} S(n,k) x^k D^k f(x)$ for some numbers $S(n,k)$, and then show that these numbers satisfy (\ref{eq-Stirling-recurrence}) (with ${\cdot \brace \cdot\cdot}$ replaced everywhere by $S(\cdot,\cdot\cdot)$).

More generally, for each word $w$ in alphabet $\{x,D\}$, with $m$ $x$'s and $n$ $D$'s, we have a unique expansion of the form
\begin{equation} \label{eq-weyl_general}
w = x^{m-n} \sum_{k \geq 0} S_w(k) x^kD^k,
\end{equation}
with both sides being viewed as operators on a space of infinitely differentiable functions. One may easily verify (\ref{eq-weyl_general}) by induction on the length of the word $w$. Uniqueness comes from considering two different expressions of the form of the right-hand side of (\ref{eq-weyl_general}), and applying their difference to $f(x)=x^{k_0}$, where $k_0$ is the smallest index $k$ for which the coefficients of $x^kD^k$ differ between the two expressions. The result is a power series in $x$ with at least one non-zero coefficient, and so cannot be identically zero.

The integer sequence $\left(S_w(k)\right)_{k \geq 0}$ that arises from (\ref{eq-weyl_general}) is what we call the {\em Stirling sequence (of the second kind)} of $w$; Scherk's identity (\ref{eq-Stirling-as-op}) states that if $w=\left(xD\right)^n$ then the Stirling sequence of $w$ coincides with the ordinary Stirling sequence of the second kind.

The set of words on alphabet $\{x,D\}$ forms a representation of the {\em Weyl algebra}, as the operators $x$ and $D$ satisfy the Weyl algebra's defining relation $Dx=xD+1$. In this context the right-hand side of (\ref{eq-weyl_general}) is referred to as the {\em normal order} of the word $w$. The normal order problem arises in quantum mechanics, where $x$ is viewed as a ``creation operator'' and $D$ as an ``annihilation operator'' in a space of polynomials. Because these operators do not commute, it is desirable from a computational point of view to find expansions of words that are presented as a sum of words, all of which have the annihilation part completely to the right and the creation part completely to the left. See \cite{BlasiakHorzelaPensonSolomonDuchamp} for an introduction to this perspective.

The study of $S_w(k)$ has a long history. Some instances (beyond $w=(xD)^n$) were studied by Scherk in his 1823 thesis \cite{Scherk}. Carlitz \cite{Carlitz,Carlitz2} derived summation formulae and identities for some instances in the 1930's, and Comtet \cite{Comtet} considered the case $w=(x^rD)^n$ in the 1970's, as did Lang \cite{Lang} in 2000.

All of these references deal with the problem through generating functions and recurrences. In 1973 Navon \cite{Navon} provided a lovely combinatorial interpretation of $S_w(k)$ for all $w$, associating a Ferrers board to $w$ and realizing $S_w(k)$ as the number of ways of placing non-attacking rooks on the board (see Section \ref{sec-H} for more details); Varvak thoroughly explored this interpretation, and obtained $q$-analogs for it, in \cite{Varvak}. Very recently Codara et al. \cite{CodaraDantonaHell} gave a combinatorial interpretation in terms of graph coloring in the case $w=(x^sD^s)^n$; our Theorem \ref{thm-main_comb_int}, which was developed independently, generalizes this interpretation to arbitrary $w$.

The arrival of the quantum mechanics community to the problem in the early 2000's has led to a flurry of activity. Blasiak et al. \cite{BlasiakPensonSolomon,BlasiakPensonSolomon2} thoroughly studied identities and recurrences for $S_w(k)$ for certain words $w$, and Schork \cite{Schork} looked at $q$-analogs. Mendez et al. \cite{MendezBlasiakPenson} looked at $S_w(k)$ for general $w$, and gave a combinatorial interpretation in terms of certain generalized tree structures.
This interpretation was explored more by Lang \cite{Lang2} and recently quite thoroughly by Blasiak and Flajolet \cite{BlasiakFlajolet}. The connection to rook polynomials was revisited by Blasiak et al. in \cite{BlasiakDuchampHorzelaPensonSolomon}, and in \cite{SolomonDuchampBlasiakHorzelaPenson} the same authors explored connections to Feynman diagrams. In \cite{AsaklyMansourSchork}, Asakly et al. showed how to read off the normal order of a word from a certain labeled tree, and in \cite{MaMansourSchork} Ma et al. explored connections between normal ordering and context-free grammars.

One contribution of the present paper is a simple new graph theoretic interpretation of $S_w(k)$ for any $w$ and $k$, that very naturally generalizes the standard interpretation of ${n \brace k}$ as a count of partitions. Specifically, we show how to associate to any $w$ a ground-set and a set of forbidden pairs (which we encode as a graph on the ground-set), in such a way that $S_w(k)$ counts the number of partitions of the ground set that avoid putting both members of a forbidden pair into the same block. The statements of all our results can be found in Section \ref{sec-statements}, and the remaining sections are devoted to the proofs.

\section{Statement of results} \label{sec-statements}

Here we outline the main results of the present paper.

\subsection{A new combinatorial interpretation of $S_w(k)$} \label{subsec-combint1}

We begin by giving a new combinatorial interpretation of $S_w(k)$. We associate to each $w$ a graph $G_w$ with the property that $S_w(k)$ enumerates the partitions of the vertex set of $G_w$ into a specified number of non-empty classes, with the property that no class includes both endvertices of an edge of $G_w$. In the case $w=(xD)^n$, $G_w$ turns out to be the empty graph on $n$ vertices and we recover the usual combinatorial interpretation of ${n \brace k}$. In the case $w=(x^sD^s)^n$, $G_w$ is the disjoint union of $n$ copies of the complete graph on $s$ vertices, and we recover a recent result of Codara et al. \cite[Proposition 2.2]{CodaraDantonaHell}. To the best of our knowledge, ours is the first combinatorial interpretation of $S_w(k)$ for arbitrary $w$ as a count of (restricted) partitions.

To define $G_w$, we first introduce the notion of a {\em Dyck word}.
\begin{defn}
A word $w$ in alphabet $\{x,D\}$ is a {\em Dyck word} if it satisfies the following:
\begin{enumerate}
\item it has the same number of $x$'s as $D$'s, and
\item reading the word from left to right, every initial segment has at least as many $x$'s as $D$'s.
\end{enumerate}
We say that a Dyck word is {\em irreducible} if either it is the word $xD$ or it is of the form $xw'D$ with $w'$ a non-empty Dyck word, and we say that it is {\em reducible} otherwise.
\end{defn}
Observe that a reducible word may be written (in unique way) as $w_1 \ldots w_\ell$, where each $w_i$ is irreducible.

A {\em Dyck path} in ${\mathbb R}^2$ is a staircase path (a path that proceeds by taking unit steps, either in the positive $x$ direction or the positive $y$ direction) that starts at $(0,0)$, ends on the line $x=y$, any never goes below this line. There is a natural correspondence between Dyck paths and Dyck words, given by mapping steps in the positive $y$ direction to $x$, and steps in the positive $x$ direction to $D$ (see Figure \ref{fig-Dyckpathexample} for an example). Irreducible Dyck words correspond to Dyck paths that meet the line $x=y$ only at their initial and terminal points, and reducible Dyck words correspond to Dyck paths that meet the line at some intermediate points as well.

We now associate to each Dyck word $w$ an unlabeled graph $G_w$ inductively, as follows:
\begin{enumerate}
\item If $w=xD$, then $G_w = K_1$ (the isolated vertex).
\item If $w$ is irreducible with $w=xw'D$ for some non-empty Dyck word $w'$, then $G_w = G_{w'} + K_1$ (the graph obtained from $G_{w'}$ by adding a dominating vertex).
\item If $w$ is reducible, say $w=w_1\ldots w_\ell$ with each $w_i$ irreducible, then $G_w = G_{w_1} \cup \cdots \cup G_{w_\ell}$ (the disjoint union of the $G_{w_i}$'s).
\end{enumerate}
In particular clause 3 above tells us that if $w=w'w''$ with $w'$ and $w''$ both Dyck words (not necessarily irreducible) then $G_w = G_{w'} \cup G_{w''}$.

We note in passing that $G_w$ belongs to the well-known family of {\em quasi-threshold}, or {\em trivially perfect} graphs (see e.g. \cite{Jing-HoJer-JeongChang} for a survey); recall that the family of quasi-threshold graphs is the smallest family that contains $K_1$ and is closed under adding dominating vertices and taking disjoint unions (if unions that do not amount to the addition of an isolated vertex are forbidden, we obtain the smaller class of {\em threshold} graphs).

As an example, consider the word $w=xxDxxDxDDD$. Writing this as
$$
x\left(\left[xD\right]\left[x\left\{\langle xD \rangle \langle xD \rangle\right\}D\right]\right)D,
$$
we quickly see that $G_w$ is constructed by taking two isolated vertices (say $a$ and $b$), adding a dominating vertex (say $c$), taking the union of the resulting graph with an isolated vertex (say $d$), and then adding a final dominating vertex (say $e$).  (See Figure \ref{fig-Dyckpathexample}.) 

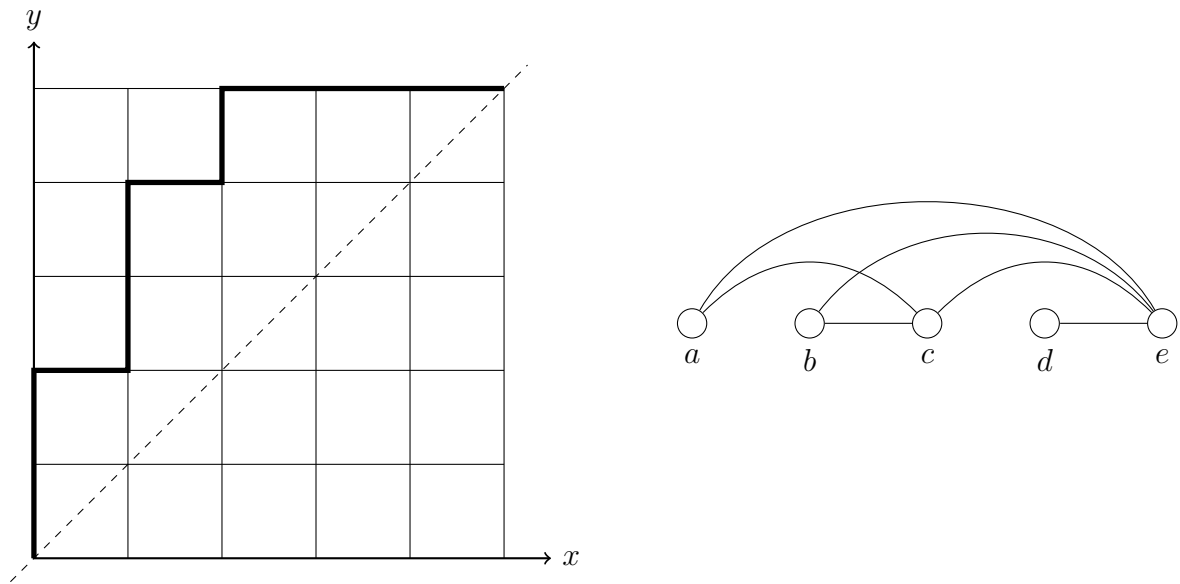
\begin{figure}[ht]
\begin{center}
\begin{tikzpicture}[scale=1.25]
    \draw (1,0) -- (1,5);
    \draw (2,0) -- (2,5);
    \draw (3,0) -- (3,5);
    \draw (4,0) -- (4,5);
    \draw (5,0) -- (5,5);
    \draw (0,1) -- (5,1);
    \draw (0,2) -- (5,2);
    \draw (0,3) -- (5,3);
    \draw (0,4) -- (5,4);
    \draw (0,5) -- (5,5);
    \draw [dashed] (-.25,-.25) -- (5.25,5.25);
    \draw [line width=2pt] (0,0) -- (0,2) -- (1,2) -- (1,4) -- (2,4) -- (2,5) -- (5,5);
    \draw [<->,thick] (0,5.5) node (yaxis) [above] {$y$}
    |- (5.5,0) node (xaxis) [right] {$x$};

    \node (v1) at (7,2.5) [circle,draw,label=270:$a$] {};
    \node (v2) at (8.25,2.5) [circle,draw,label=270:$b$] {};
    \node (v3) at (9.5,2.5) [circle,draw,label=270:$c$] {};
    \node (v4) at (10.75,2.5) [circle,draw,label=270:$d$] {};
    \node (v5) at (12,2.5) [circle,draw,label=270:$e$] {};
    \draw (v2) -- (v3);
    \draw (v4) -- (v5);
    \path[every node/.style={font=\sffamily\small}] (v1) edge[loop, distance=.4in, out=45, in=135] node {} (v3);
    \path[every node/.style={font=\sffamily\small}] (v1) edge[loop, distance=.7in, out=60, in=120] node {} (v5);
    \path[every node/.style={font=\sffamily\small}] (v2) edge[loop, distance=.55in, out=52.5, in=127.5] node {} (v5);
    \path[every node/.style={font=\sffamily\small}] (v3) edge[loop, distance=.4in, out=45, in=135] node {} (v5);

\end{tikzpicture}

\caption{The Dyck path corresponding to the Dyck word $w = xxDxxDxDDD$ and the associated graph $G_w$.}
\label{fig-Dyckpathexample}
\end{center}
\end{figure}

To state our first theorem, it is convenient to generalize the symbol ${n \brace k}$ to graphs.
\begin{defn}
Let $G$ be a graph and $k \geq 0$ an integer. The {\em $k^{th}$ graph Stirling number} of $G$, denoted ${G \brace k}$, is the number of ways of partitioning the vertex set of $G$ into $k$ non-empty classes, none of which contains both endvertices of an edge of $G$.
\end{defn}
Equivalently ${G \brace k}$ counts partitions of $G$ into $k$ non-empty {\em independent sets}, that is, sets of mutually non-adjacent vertices; it also counts the number of different proper $k$-colorings of $G$ using all $k$ colors, with two colorings considered to be the same if they differ only on the names of the colors. These numbers were first studied in their own right (to the best of our knowledge) by Tomescu \cite{Tomescu}, although they were probably first introduced (in the setting of planar graphs) by Birkhoff \cite{Birkhoff}. Note that ${G \brace k}$ does indeed generalize the Stirling numbers ${n \brace k}$, since when $G=E_n$, the graph on $n$ vertices with no edges, we have ${E_n \brace k} = {n \brace k}$.

We can now state of first theorem, which gives a simple new combinatorial interpretation of $S_w(k)$ in terms of restricted partitions.
\begin{thm} \label{thm-main_comb_int}
Let $w$ be a Dyck word in the alphabet $\{x,D\}$. For all $k \geq 0$ we have
$$
S_w(k) = {G_w \brace k}.
$$
\end{thm}
Our proof of Theorem \ref{thm-main_comb_int} will be direct, in the sense that we do not rely on previously known formulae or combinatorial interpretations for $S_w(k)$.

As an illustration of Theorem \ref{thm-main_comb_int}, let us return to $xxDxxDxDDD$. A little computation with the relation $Dx=xD+1$ yields
$$
xxDxxDxDDD = 2x^3D^3 + 4x^4D^4 + x^5D^5.
$$
On the other hand, we have ${G_w \brace 3} = 2$ (the two partitions of $V(G_w)$ into three non-empty independent sets are $ab|cd|e$ and $abd|c|e$), ${G_w \brace 4} = 4$ (the partitions being $ab|c|d|e$, $ad|b|c|e$, $a|bd|c|e$ and $a|b|cd|e$), ${G_w \brace 5} = 1$ (the unique partition being $a|b|c|d|e$), and ${G_w \brace k} = 0$ for all other $k$, exactly as predicted by Theorem \ref{thm-main_comb_int}.

In the case $w=(x^{t_i}D^{t_i})^n$ we get a particularly appealing combinatorial interpretation. Here $G_w=K_{t_1} \cup \ldots \cup K_{t_n}$, the disjoint union of cliques of various sizes, and Theorem \ref{thm-main_comb_int} gives that $S_w(k)$ counts the number of partitions of this union of cliques into $k$ non-empty independent sets. The case when all $t_i=2$ was observed in \cite{BlasiakFlajolet}, in slightly different language, and the case of all $t_i=t$ for general $t$ has appeared recently in \cite{CodaraDantonaHell}. Note that in this last case we may whimsically interpret ${G_w \brace k}$ as the number of ways of breaking up a gathering of $n$ sets of $t$-tuplets into $k$ non-empty blocks, in such a way that no block contains more than one member from each set of $t$-tuplets. Variants of this problem for twins, or $t=2$, were considered by Griffiths in \cite{Griffiths}; our work on these notes began after reading that paper.

Theorem \ref{thm-main_comb_int} can easily be extended to give a combinatorial interpretation of $S_w(k)$ for arbitrary $w$.
\begin{defn} \label{defn-associated-word}
Let $w$ be a word in the alphabet $\{x,D\}$. Let $a=a(w)$ be the least non-negative integer such that all initial segments of $x^aw$ have at least as many $x$'s as $D$'s, and let $b=b(w)$ be the unique non-negative integer such that $x^awD^b$ is a Dyck word. We refer to $x^awD^b$ as the Dyck word {\em associated with} $w$.
\end{defn}
The normal order of $w$ is easily obtained from that of $x^awD^b$. Indeed, suppose that $w$ has $m$ $x$'s and $n$ $D$'s. From (\ref{eq-weyl_general}) we have $w = x^{m-n}\sum_{k \geq 0} S_w(k) x^k D^k$. Then
\begin{eqnarray*}
x^awD^b & = & x^a\left(x^{m-n} \sum_{k \geq 0} S_w(k) x^k D^k\right)D^b \\
& = & x^{a+m-n} \sum_{k \geq 0} S_w(k) x^k D^{k+b} \\
& = & \sum_{k \geq 0} S_w(k) x^{k+b}D^{k+b},
\end{eqnarray*}
the last equality using $a+m=b+n$ (valid since $x^awD^b$ is a Dyck word). Since also $x^awD^b = \sum_{k \geq 0} S_{x^awD^b}(k) x^k D^k$, we get the identity $S_w(k)=S_{x^awD^b}(k+b)$. Therefore, the following is an immediate corollary of Theorem \ref{thm-main_comb_int}.
\begin{cor} \label{cor-main_comb_int}
Let $w$ be an arbitrary word in the alphabet $\{x,D\}$, and let $x^awD^b$ be its associated Dyck word, as in Definition \ref{defn-associated-word}. For all $k \geq 0$ we have
$$
S_w(k) = {G_{x^awD^b} \brace k+b}.
$$
\end{cor}

\medskip

Theorem \ref{thm-main_comb_int} and Corollary \ref{cor-main_comb_int} yield an explicit expression for $S_w(k)$. Let $w$ be any word with, say, $m$ $x$'s, and let $x^awD^b$ be its associated Dyck word.
\begin{defn} \label{defn-height}
The {\em height} $a_i$ of the $i^{th}$ $x$ in $x^awD^b$ is the excess of $x$'s over $D$'s in the initial segment of $x^awD^b$ that ends immediately prior to the $i^{th}$ $x$.
\end{defn}
Equivalently, the height of each $x$ can be calculated from the Dyck path associated with $x^awD^b$: if the step in the positive $y$ direction corresponding to a particular $x$ of the word goes from $(a,b)$ to $(a, b+1)$, then the height of that $x$ is $b-a$.
\begin{thm} \label{thm-closedform}
Let $w$ be an arbitrary word in the alphabet $\{x,D\}$, and let $x^awD^b$ be its associated Dyck word. With the $a_i$'s as in Definition \ref{defn-height}, we have
$$
S_w(k) = \frac{1}{(k+b)!}\sum_{\ell=0}^{k+b} (-1)^\ell{k+b \choose \ell} \prod_{i=1}^{m+a} (k+b-\ell-a_i).
$$
\end{thm}
For example, if $w=xxDxxDxDDD$ then $m=5$, $a=b=0$, $a_1=0$, $a_2=1$, $a_3=1$, $a_4=2$, $a_5=2$ and
$$
S_w(k) = \frac{1}{k!}\sum_{\ell=0}^k (-1)^\ell{k \choose \ell} (k-\ell)(k-\ell-1)^2(k-\ell-2)^2.
$$
Similar explicit expressions have appeared in \cite{Varvak} and \cite{MendezBlasiakPenson}. Note that when $w=(xD)^n$, Theorem \ref{thm-closedform} immediately reduces to the familiar
$$
{n \brace k} = \frac{1}{k!}\sum_{\ell=0}^k (-1)^\ell {k \choose \ell}(k-\ell)^n.
$$

The proofs of Theorems \ref{thm-main_comb_int} and \ref{thm-closedform} appear in Section \ref{sec-comb_int_new}.

\subsection{A closely related combinatorial interpretation} \label{subsec-Hintro}

By combining a result of Navon \cite{Navon} with one of Goldman, Joichi and White \cite{GoldmanJoichiWhite}, we find another graph $H_w$ that can naturally be associated to a Dyck word $w$, such that $S_w(k) = {H_w \brace k}$ for all $k$. To define this graph, label each unit square in ${\mathbb Z}^2$ with the coordinates of its top-right corner (so, for example, the square with corners at $(0,0)$, $(1,0)$, $(0,1)$ and $(1,1)$ gets label $(1,1)$). Given a Dyck word $w$, let ${\mathcal W}_w$ be the set of (labels of) unit squares that lie below the staircase path of $w$, and completely above the line $x=y$. For example, if $w=(xD)^n$ then ${\mathcal W}_w=\emptyset$, and if $w=xxDxxDxDDD$ then ${\mathcal W}_w=\{(1,2),(2,3),(2,4),(3,4),(3,5),(4,5)\}$. Define a graph $H_w$ on vertex set $\{1, \ldots, n\}$ (where $n$ is the number of $x$'s in $w$) by putting an edge from $i$ to $j$ ($i < j$) if and only if $(i,j) \in {\mathcal W}_w$. (See Figure \ref{fig-graphHw}.)

\begin{figure}[ht!]
\begin{center}
\begin{tikzpicture}[scale=1.25]
    \draw (1,0) -- (1,5);
    \draw (2,1) -- (2,5);
    \draw (3,2) -- (3,5);
    \draw (4,3) -- (4,5);
    \draw (5,4) -- (5,5);
    \draw (0,1) -- (2,1);
    \draw (0,2) -- (3,2);
    \draw (0,3) -- (4,3);
    \draw (0,4) -- (5,4);
    \draw (0,5) -- (5,5);
    \draw [dashed] (-.25,-.25) -- (5.25,5.25);
    \draw [line width=2pt] (0,0) -- (0,2) -- (1,2) -- (1,4) -- (2,4) -- (2,5) -- (5,5);
    \draw [<->,thick] (0,5.5) node (yaxis) [above] {$y$}
    |- (5.5,0) node (xaxis) [right] {$x$};

    \node at (.5,1.5) {$\underline{(1,2)}$};
    \node at (1.5,2.5) {$(2,3)$};
    \node at (1.5,3.5) {$\underline{(2,4)}$};
    \node at (2.5,3.5) {$(3,4)$};
    \node at (2.5,4.5) {$\underline{(3,5)}$};
    \node at (3.5,4.5) {$(4,5)$};

    \node (v1) at (7,2.5) [circle,draw,label=270:$1$] {};
    \node (v2) at (8.25,2.5) [circle,draw,label=270:$2$] {};
    \node (v3) at (9.5,2.5) [circle,draw,label=270:$3$] {};
    \node (v4) at (10.75,2.5) [circle,draw,label=270:$4$] {};
    \node (v5) at (12,2.5) [circle,draw,label=270:$5$] {};
    \draw (v1) -- (v2);
    \draw (v2) -- (v3);
    \draw (v3) -- (v4);
    \draw (v4) -- (v5);
    \path[every node/.style={font=\sffamily\small}] (v2) edge[loop, distance=.5in, out=52.5, in=127.5] node {} (v4);
    \path[every node/.style={font=\sffamily\small}] (v3) edge[loop, distance=.5in, out=52.5, in=127.5] node {} (v5);

\end{tikzpicture}
\caption{The staircase path of $w$ and the graph $H_w$ for $w=xxDxxDxDDD$.}
\label{fig-graphHw}
\end{center}
\end{figure}
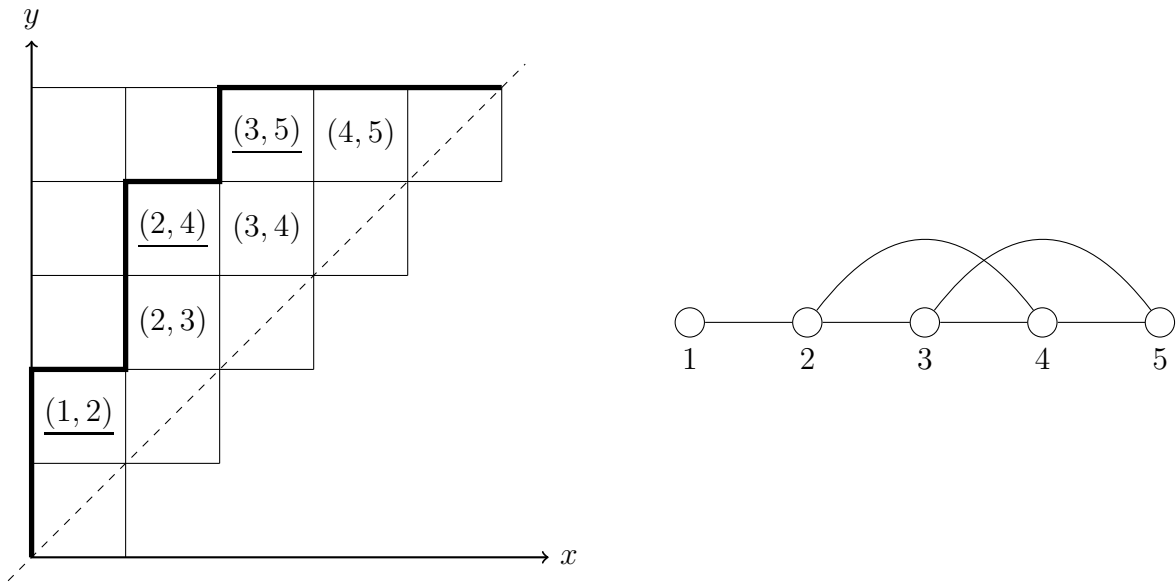

It is worth noting that $H_w$ is determined by the locations of the peaks of the Dyck path of $w$, that is, by the places where the path takes a step up followed by a step to the right. To make this precise, say that the Dyck path of $w$ {\em has a peak} at $(x,y)$ if it takes a step from $(x-1,y-1)$ to $(x-1,y)$ and then steps to $(x,y)$. Let ${\mathcal T}_w=\{(x_1,y_1), \ldots, (x_k,y_k)\}$ be the set of peaks of the path of $w$. Then it is easy to see that the edge set of $H_w$ can be covered by putting a clique on each of the consecutive segments $\{x_i, \ldots, y_i\}$, $1 \leq i \leq k$. For example, if $w=(xD)^n$ then ${\mathcal T}_w=\{(1,1),\ldots,(n,n)\}$ and the edge set of $H_w$ is empty; while if $w=xxDxxDxDDD$ then ${\mathcal T}_w=\{(1,2),(2,4),(3,5)\}$, and the edge set of $H_w$ is $\{\{1,2\}\} \cup \{\{2,3\},\{2,4\},\{3,4\}\} \cup \{\{3,4\},\{3,5\},\{4,5\}\}$, i.e., it is composed of cliques on the vertex sets $\{1,2\}$, $\{2,3,4\}$, and $\{3,4,5\}$ (see Figure \ref{fig-graphHw}, where the coordinates of the peaks of the path are underlined). 

We note in passing that $H_w$ belongs to the family of {\em indifference graphs}. Indeed, a characterization of indifference graphs mentioned in \cite{GebhardSagan} is that they are exactly those graphs $H$ on vertex set $\{v_1, \ldots, v_d\}$ for which there is some collection ${\mathcal C}$ of intervals from $\{1, \ldots, d\}$ such that the edge set of $H$ is $\{v_iv_j: i,j~\mbox{are in some element of}~{\mathcal C}\}$. Notice that the graphs $H_w$ and $G_w$ are sometimes isomorphic (for example, when $w=\prod_{i=1}^n x^{t_i}D^{t_i}$ for arbitrary $t_i$'s), but not always (for example, when $w=xxDxxDxDDD$).
\begin{thm} \label{thm-main_comb_int2}
Let $w$ be Dyck word in the alphabet $\{x,D\}$. For all $k \geq 0$ we have
$$
S_w(k) = {H_w \brace k}.
$$
\end{thm}
As in Section \ref{subsec-combint1}, the following is an immediate corollary of Theorem \ref{thm-main_comb_int2}.
\begin{cor} \label{cor-main_comb_int2}
Let $w$ be an arbitrary word in the alphabet $\{x,D\}$, and let $x^awD^b$ be its associated Dyck word. For all $k \geq 0$ we have
$$
S_w(k) = {H_{x^awD^b} \brace k+b}.
$$
\end{cor}
We give the proof of Theorem \ref{thm-main_comb_int2} in Section \ref{sec-H}, where we also give a direct proof (not using the results of \cite{Navon} and \cite{GoldmanJoichiWhite}) of the identity ${H_w \brace k}={G_w \brace k}$ for all Dyck words $w$ and $k \geq 0$.

\subsection{A new summation formula for $S_w(k)$ when $w=(x^sD^s)^n$}

All explicit expressions for $S_w(k)$ that have appeared in the literature have taken the form of alternating sums. Using Theorem \ref{thm-main_comb_int}, we can obtain a new expression for $S_w(k)$, in the special case $w=(x^sD^s)^n$, as a positive linear combination of ordinary Stirling numbers. In what follows we use $[x^\ell]p(x)$ for the coefficient of $x^\ell$ in the polynomial $p(x)$. The proof of the following theorem is given in Section \ref{sec-summ}.
\begin{thm} \label{thm-summ}
Let $w=(x^sD^s)^n$. For each $k \geq 0$ we have
$$
S_w(k) = \sum_{\ell = 0}^{(s-1)(n-1)} f(n,s,\ell) {s(n-1)+1-\ell \brace k-(s-1)}
$$
where
\begin{equation} \label{fnsl}
f(n,s,\ell) = \sum_{i_1+\ldots +i_{s-1} = \ell} \binom{n-1}{i_1} \binom{2(n-1)-i_1}{i_2} \ldots \binom{(s-1)(n-1)-i_1-\ldots -i_{s-2}}{i_{s-1}}.
\end{equation}
\end{thm}

In Section \ref{sec-summ} we also establish the following alternate expression for $f(n,s,\ell)$:
\begin{equation} \label{fnsl2}
f(n,s,\ell) = [x^\ell]\left((1+x)(1+2x)\ldots (1+(s-1)x)\right)^{n-1}.
\end{equation}
The (unsigned) Stirling number of the first kind ${a \brack b}$ counts the number of permutations of $a$ symbols that decompose into exactly $b$ cycles. Using a well known identity satisfied by the Stirling numbers of the first kind, (\ref{fnsl2}) immediately gives the following nice connection between generalized Stirling numbers of the second kind, and ordinary Stirling numbers of the first kind:
$$
f(n,s,\ell) = [x^\ell]\left(\sum_{j=0}^{s-1} {s \brack s-j} x^j \right)^{n-1}.
$$

The {\em Bell number} $B(w)$ of a word $w$ is defined by $B(w) = \sum_k S_w(k)$ (so the Bell number of the word $(xD)^n$ is $B_n$, the $n^{th}$ ordinary Bell number, counting the number of partitions of a set of size $n$ into non-empty classes). From Theorem \ref{thm-summ} we easily obtain
\begin{equation} \label{eq-Bell}
B((x^sD^s)^n) = \sum_{\ell = 0}^{(s-1)(n-1)} f(n,s,\ell) B_{s(n-1)+1-\ell}.
\end{equation}
In \cite{BlasiakPensonSolomon2} the comment is made that the Bell number $B((x^sD^s)^n)$ can always be expressed in terms of conventional Bell
numbers and $r$-nomial (binomial, trinomial, $\ldots$) coefficients, and the illustrative example $B((x^2D^2)^n) = \sum_{\ell = 0}^{n-1} \binom{n-1}{\ell} B_{2n-1-\ell}$ is given; (\ref{eq-Bell}) illustrates this comment explicitly for arbitrary $s$.

\subsection{Increasing forests}\label{sec-introforests}

An {\em $r$-ary tree} is a tree in which every vertex, including a designated root, has some number $i$ ($0 \leq i \leq r$) of children, with the set of children equipped with a bijection to some subset of a fixed set of size $r$ (when $r=2$ this set is often taken to be $\{\mbox{left},\mbox{right}\}$, for example, and for $r=3$ it might be $\{\mbox{left},\mbox{middle},\mbox{right}\}$; for general $r$ we take it to be $\{1, \ldots r\}$). An {\em $r$-ary forest} is a forest in which each component is an $r$-ary tree. An {\em increasing $r$-ary forest} is an $r$-ary forest on, say, $n$ vertices, together with a bijection from the vertices to $\{1, \ldots, n\}$, with the property that the labels go in increasing order when read along any path starting from a root vertex of a component.

Let $F(r,n,k)$ denote the set of increasing $r$-ary forests with $n$ vertices and $k$ components. It is easy to see that $|F(1,n,k)| = {n \brace k}$; in other words, writing $w(r,n)$ for the word $(x^rD)^n$, we have $|F(1,n,k)| = S_{w(1,n)}(k)$. More generally, Mendez et al \cite[Section 5]{MendezBlasiakPenson} have shown that the Stirling sequence of $(x^rD)^n$ enumerates increasing $r$-ary forests with $n$ vertices by number of components; specifically, for all $n$, $r$ and $k$,
\begin{equation} \label{M-for}
|F(r,n,k)| = S_{w(r,n)}(k).
\end{equation}

Corollaries \ref{cor-main_comb_int} and \ref{cor-main_comb_int2} provide alternate combinatorial interpretations of $S_w(k)$, in the case $w=w(r,n)$, that are quite appealing. Indeed, in this case $G_w$ is simply the threshold graph obtained from the graph on no vertices by iterating $n$ times the operation of adding an isolated vertex and then adding $r-1$ dominating vertices, and $H_w$ is the graph on vertex set $\{1, \ldots, rn\}$, with edges covered by cliques on vertices $i$ through $ir$, for $1 \leq i \leq n$ (see Figures \ref{fig-G(n,r)}, \ref{fig-Gw} and \ref{fig-Hw} in Section \ref{sec-forests}). We refer to these graphs as $G(n,r)$ and $H(n,r)$ respectively. The following identities follow immediately from Theorems \ref{thm-main_comb_int} and \ref{thm-main_comb_int2}, via (\ref{M-for}):
\begin{eqnarray*}
{G(n,r) \brace k + (r-1)n} & = & |F(r,n,k)| \\
{H(n,r) \brace k + (r-1)n} & = & |F(r,n,k)|.
\end{eqnarray*}
In Section \ref{sec-forests} we give combinatorial proofs of both of these identities, by exhibiting bijections from the set of increasing $r$-ary forests with $n$ vertices and $k$ components to the set of partitions of both $G(n,r)$ and $H(n,r)$ into $k + (r-1)n$ non-empty independent sets.

\subsection{$q$-analogs} \label{subsec-q}

The ordinary Weyl algebra on alphabet $\{x,D\}$ is generated by the relation $Dx=xD+1$. The {\em $q$-deformed Weyl algebra} is instead generated by the relation $Dx=qxD+1$ (where $q$ is an indeterminate). This relation has been studied, for example, in the context of quantum harmonic oscillators \cite{ArikCoon}. If $w$ is a word in the $q$-deformed Weyl algebra, with $m$ $x$'s and $n$ $D$'s, then we have the following analog of the normal order equation (\ref{eq-weyl_general}) (again easily verified by induction, and again unique)
\begin{equation} \label{eq-weyl_general_q}
w = x^{m-n} \sum_{k \geq 0} S^q_w(k) x^kD^k,
\end{equation}
where now $S^q_w(k)$ is a polynomial in $q$ with non-negative integer coefficients, that evaluates to $S_w(k)$ when $q=1$.

In the case $w=(xD)^n$, (\ref{eq-weyl_general_q}) leads to a $q$-analog of the Stirling numbers of the second kind. From a combinatorial perspective, a more natural way to define a $q$-analog of the Stirling numbers is through the recurrence
$$
{n \brace k}_q = \left\{
\begin{array}{ll}
q^{k-1}{n-1 \brace k-1}_q + [k]_q{n-1 \brace k}_q & \mbox{if both $n>0$ and $k>0$, and} \\
{\bf 1}_{\{n=k\}} & \mbox{if $nk=0$}
\end{array}
\right.
$$
where $[k]_q$ is, as usual, the polynomial $1+q+\ldots q^{k-1}$. This formulation goes back to Carlitz \cite{Carlitz3}.
Happily, these two refinements of the Stirling numbers coincide: for all $n$, $k$ and $q$ we have $S^q_{(xD)^n}(k) = {n \brace k}_q$ (with the proof of this following the lines of the proof of (\ref{eq-Stirling-as-op})).

Varvak \cite{Varvak} extended Navon's combinatorial interpretation of $S_w(k)$ for arbitrary $w$ to a combinatorial interpretation of (the coefficients of) $S^q_w(k)$. We are also able to extend our interpretation (from Section \ref{subsec-combint1}) to the realm of the $q$-deformed Weyl algebra. Specifically, given a Dyck word $w$, by Theorem \ref{thm-main_comb_int} there is a quasi-threshold graph $G_w$ with the property that ${G_w \brace k} = S_w(k)$. Let ${\mathcal P}(w,k)$ be the set of partitions of $G_w$ into $k$ non-empty independent sets. We show now how to associate a weight ${\rm wt}(P)$ to each $P \in {\mathcal P}(w,k)$ so that $\sum_{P \in {\mathcal P}(w,k)} q^{{\rm wt}(P)} = S_w^q(k)$.

Not unexpectedly, given the inductive construction of $G_w$, the definition of ${\rm wt}(P)$ will be inductive. It will also depend on a predetermined total order on the set of independent sets of $G_w$ (which may be chosen arbitrarily); note that such a total order naturally induces a total order on the set of independent sets of any subgraph of $G_w$. Before giving the definition, we make some preliminary observations.

For $w=xD$, the only value of $k$ for which ${\mathcal P}(w,k) \neq \emptyset$ is $k=1$, and there is a unique partition in ${\mathcal P}(w,1)$.
If $w=xw'D$ is irreducible, then $G_w$ can be written as $G_{w'}+K_1$. Let $v$ be the dominating vertex in this decomposition. For each $k$ for which ${\mathcal P}(w,k) \neq \emptyset$, each $P \in {\mathcal P}(w,k)$ must include $v$ as a singleton part, and the map $\varphi: {\mathcal P}(w,k) \rightarrow {\mathcal P}(w',k-1)$ that removes that singleton part is a bijection.

If $w=w_1w_2\ldots w_\ell$ is reducible (with each $w_i$ irreducible), then $G_w$ can be written as the disjoint union of $G_{w_1}$ and $G_{w_2\ldots w_\ell}$. Fix a $k$ for which ${\mathcal P}(w,k) \neq \emptyset$, and consider $P \in {\mathcal P}(w,k)$. There are numbers $r, s$ such that exactly $r$ of the parts of $P$ have non-empty intersection with $G_{w_1}$, exactly $s$ of them have non-empty intersection with $G_{w_2\ldots w_\ell}$, and exactly $(r+s)-k$ of them have non-empty intersection with both $G_{w_1}$ and $G_{w_2\ldots w_\ell}$. By projection $P$ induces partitions $P_1 \in {\mathcal P}(w_1,r)$ with parts $x_1, \ldots, x_r$ (written in increasing order with respect to the total order on independent sets) and $P_2 \in {\mathcal P}(w_2\ldots w_\ell,s)$ with parts $y_1, \ldots, y_s$ (also in increasing order). We may encode how $P$ is constructed from $P_1$ and $P_2$ using an $r$ by $s$ matrix whose $ij$ entry is $1$ if $x_i \cup y_j$ is one of the parts of $P$, and $0$ otherwise. Call this matrix $M$; note that it contains a dimension $(r+s)-k$ permutation matrix as a minor, and all other entries are $0$. When we come to define ${\rm wt}(P)$ 
in the reducible case, we will need to associate a weight to this matrix $M$.

\begin{defn} \label{def-f(M)}
Let $M$ be an $r$ by $s$ matrix that contains a dimension $(r+s)-k$ permutation matrix as a minor, and has all other entries $0$. Mark all of the $0$'s in $M$ that are either below a $1$ (in the same column) or to the right of a $1$ (in the same row). The {\em weight} of $M$, denoted $f(M)$, is the number of unmarked $0$'s in $M$.
\end{defn}
See Figure \ref{fig-q analog} in Section \ref{sec-q} for an example $M$.

We are now in a position to define the weight of a partition. For $w$ a Dyck word, and $P$ a partition of $G_w$ into non-empty independent sets, associate a weight ${\rm wt}(P)$ to $P$ inductively as follows.
\begin{enumerate}
\item If $w=xD$, then ${\rm wt}(P)=0$.
\item If $w$ is irreducible with $w=xw'D$ for some non-empty Dyck word $w'$, then ${\rm wt}(P) = {\rm wt}(\varphi(P))$.
\item If $w$ is reducible, say $w=w_1\ldots w_\ell$ with each $w_i$ irreducible, then
$$
{\rm wt}(P) = {\rm wt}(P_1) + {\rm wt}(P_2) + f(M).
$$
\end{enumerate}

We prove the following result in Section \ref{sec-q}.

\begin{thm} \label{thm-q}
For each Dyck word $w$ and each $k\geq 0$, we have
$$
\sum_{P \in {\mathcal P}(w,k)} q^{{\rm wt}(P)} = S_w^q(k).
$$
\end{thm}

While the construction above may seem involved, the following example illustrates a subtlety of $q$-analogs that must be captured by any interpretation. For $w=xDxDxD$, we have $S_{w}^q(2) = 2q+q^2$, despite the fact that all three partitions of the vertices of $G_w=E_3$ into two nonempty independent sets are isomorphic.

\section{Interpreting $S_w(k)$ in terms of partitions of $G_w$} \label{sec-comb_int_new}

We begin with the proof of Theorem \ref{thm-main_comb_int}, which depends on the following two claims.
\begin{claim} \label{claim-dom}
Let $w'$ be a word in the alphabet $\{x, D\}$, and let $G'$ be a graph with the property that for all $k \geq 0$,
$$
S_{w'}(k) = {G' \brace k}.
$$
Let $w=xw'D$ and let $G$ be obtained from $G'$ by adding a dominating vertex.
For all $k \geq 0$,
$$
S_w(k) = {G \brace k}.
$$
\end{claim}

\begin{proof}
Using
$$
w' f(x) = \sum_{k \geq 0} {G' \brace k} x^k D^k f(x)
$$
we easily get (applying the above with $f(x)$ replaced by $f'(x)$, and setting ${G' \brace -1}=0$)
$$
xw'D f(x) = \sum_{k \geq 0} {G' \brace k-1} x^k D^k f(x).
$$
The proof is completed by noting that ${G \brace k} = {G' \brace k-1}$ for all $k \geq 0$, since in any partition of $G$ into non-empty independent sets, the dominating vertex added in going from $G'$ to $G$ must form a singleton block.
\end{proof}

\begin{claim} \label{claim-union}
Let $w_1, w_2$ be words in the alphabet $\{x, D\}$, and let $G_1$, $G_2$ be graphs with the property that for each $i \in \{1,2\}$ and all $k \geq 0$,
\begin{equation} \label{base2}
S_{w_i}(k) = {G_i \brace k}.
\end{equation}
For all $k \geq 0$,
$$
S_{w_1w_2}(k) = {G_1 \cup G_2 \brace k}
$$
where $w_1w_2$ is the concatenation of $w_1$ and $w_2$, and $G_1 \cup G_2$ is the disjoint union of $G_1$ and $G_2$.
\end{claim}

\begin{proof}
Using (\ref{base2}) we have, for arbitrary $f$,
\begin{eqnarray}
w_1w_2 f(x) & = & w_1~\sum_{k_2 \geq 0} {G_2 \brace k_2} x^{k_2} D^{k_2}f(x) \nonumber \\
& = & \sum_{k_1 \geq 0} {G_1 \brace k_1} x^{k_1} D^{k_1} \sum_{k_2 \geq 0} {G_2 \brace k_2} x^{k_2} D^{k_2}f(x) \nonumber \\
& = & \sum_{k_1, k_2 \geq 0} {G_1 \brace k_1}{G_2 \brace k_2} x^{k_1} D^{k_1} x^{k_2} D^{k_2}f(x) \label{first}.
\end{eqnarray}
Now by Leibniz' rule
($(fg)^{(n)}(x)=\sum_{k \geq 0} \binom{n}{k} f^{(k)}(x)g^{(n-k)}(x)$) for the iterated derivative of a product we have
\begin{eqnarray}
x^{k_1} D^{k_1} x^{k_2} D^{k_2}f(x) & = & \sum_{j \geq 0} \frac{k_1^{\underline j}k_2^{\underline j}}{j!} x^{k_1+k_2-j}D^{k_1+k_2-j}f(x). \label{Leibnitz_q=1}
\end{eqnarray}
(Recall that $x^{\underline j}$ is the $j^{th}$ falling power of $x$, that is, $x(x-1)\ldots (x-j+1)$.)
Inserting this into (\ref{first}) and extracting the coefficient of $x^kD^k$ from each side we get
\begin{eqnarray}
S_{w_1w_2}(k) & = & \sum_{k_1, k_2 \geq 0} {G_1 \brace k_1}{G_2 \brace k_2} \frac{k_1^{\underline{k_1+k_2-k}}k_2^{\underline{k_1+k_2-k}}}{(k_1+k_2-k)!} \nonumber \\
& = & \sum_{k_1, k_2 \geq 0} {G_1 \brace k_1}{G_2 \brace k_2} \binom{k_1}{k_1+k_2-k} \binom{k_2}{k_1+k_2-k} (k_1+k_2-k)!. \label{induction}
\end{eqnarray}
We claim that the right-hand side of (\ref{induction}) is exactly ${G_1 \cup G_2 \brace k}$. Indeed, one way to generate all partitions of $G_1 \cup G_2$ into $k$ nonempty independent sets is as follows. First, fix a pair $(k_1, k_2)$. Then, for each $i$, partition $G_i$ into $k_i$ non empty independent sets (there are ${G_1 \brace k_1}{G_2 \brace k_2}$ ways to do this). Next, choose $k_1+k_2-k$ of the classes from $G_1$ and $k_1+k_2-k$ of the classes from $G_2$ (there are $\binom{k_1}{k_1+k_2-k} \binom{k_2}{k_1+k_2-k}$ ways to do this). Finally, merge the chosen classes in pairs, one from $G_1$ and one from $G_2$ (there are $(k_1+k_2-k)!$ ways to do this), thereby creating $k_1 + k_2 - (k_1+k_2-k) = k$ non-empty independent sets.
\end{proof}

We are now ready to prove Theorem \ref{thm-main_comb_int}.

\begin{proof} (Theorem \ref{thm-main_comb_int}) We proceed by induction on the length of $w$. If $w$ has length $2$ then $w=xD$ and $G_w=K_1$ and the result is trivial.

If $w$ is irreducible, and of length greater than $2$, then $w=xw'D$ for some Dyck word $w$ which (by induction) has an associated graph $G_{w'}$, constructed as described in Section \ref{subsec-combint1}, with $S_{w'}(k)={G_{w'} \brace k}$ for all $k \geq 0$. That $S_w(k) = {G_w \brace k}$ for all $k \geq 0$, where $G_w$ is obtained from $G_{w'}$ by adding a dominating vertex follows from Claim \ref{claim-dom}.

If $w$ is reducible, and of length greater than $2$, then $w=w_1w_2 \ldots w_k$ for some (irreducible) Dyck words $w_i$ which (by induction) have associated graphs $G_{w_i}$, constructed as described in Section \ref{subsec-combint1}, with $S_{w_i}(k)={G_{w_i} \brace k}$ for all $k \geq 0$. That $S_w(k) = {G_w \brace k}$ for all $k \geq 0$ where $G_w$ is the disjoint union of the $G_{w_i}$'s follows from repeated applications of Claim \ref{claim-union}.
\end{proof}

For Theorem \ref{thm-closedform} we utilize the chromatic polynomial. Recall that associated to a graph $G$ there is a polynomial $\chi_G(q)$, the {\em chromatic polynomial}, whose value at each positive integer $q$ is the number of proper $q$-colorings of $G$, that is, the number of functions $f:V\rightarrow \{1,\ldots, q\}$ satisfying $f(u)\neq f(v)$ whenever $uv \in E$. A key observation is that for all $G$, $\chi_G(q)$ determines $({G \brace k})_{k \geq 0}$, and vice-versa. Indeed, on the one hand inclusion-exclusion gives
\begin{equation} \label{inc-exc}
{G \brace k} = \frac{1}{k!}\sum_{i=0}^k (-1)^i{k \choose i}\chi_G(k-i),
\end{equation}
while on the other hand
$$
\chi_G(q) = \sum_{k \geq 0} {G \brace k}q^{\underline k}.
$$
To see this second relation, note that given a palette of $q$ colors, for each $k$ there are ${G \brace k}$ ways to partition the vertex set into $k$ non-empty color classes, and $q^{\underline k}$ ways to assign colors the classes. A particular consequence of (\ref{inc-exc}) that we will use later is that
\begin{equation} \label{lem-chromaticpoly}
\mbox{if $\chi_G(q)=\chi_{G'}(q)$ for all $q$ then ${G \brace k}={G' \brace k}$ for all $k \geq 0$}.
\end{equation}

Theorem \ref{thm-closedform} follows immediately from (\ref{inc-exc}) and the following claim that expresses the chromatic polynomial of $G_w$ in terms of the heights of the $x$'s in $w$ (recall Definition \ref{defn-height} for the definition of height). 
\begin{claim} \label{claim-chrom}
For any Dyck word $w$, with $m$ $x$'s having heights $a_1, \ldots, a_m$,
$$
\chi_{G_w}(q) = \prod_{i=1}^m (q-a_i).
$$
\end{claim}

\begin{proof}
We proceed by induction on the length of $w$, with length 2 trivial. Consider now a word $w$ of length at least $4$. If $w$ is reducible, say $w=w_1\ldots w_k$ with each $w_i$ an (irreducible) Dyck word, then by Claim \ref{claim-union} we have $G_w = G_{w_1} \cup \ldots \cup G_{w_k}$ and so $\chi_{G_w}(q) = \prod_{i=1}^k \chi_{G_{w_i}}(q)$. The height of an $x$ in $w_i$ is the same as the height of the corresponding $x$ in $w$, and so the claim follows by induction. If instead $w=xw'D$ is irreducible then by Claim \ref{claim-dom}, $G_w$ is obtained from $G_{w'}$ by adding a dominating vertex, so $\chi_{G_w}(q) = q\chi_{G_{w'}}(q-1)$. The height of an $x$ in $w'$ is now one less than the height of the corresponding $x$ in $w$, and so again the claim follows by induction.
\end{proof}

\section{Interpreting $S_w(k)$ in terms of partitions of $H_w$} \label{sec-H}

To prove Theorem \ref{thm-main_comb_int2}, we need only combine two old results. The first is due to Goldman et al., and forms part of their series of results on rook polynomials. An {\em $n$-board} is a subset of $\{1, \ldots, n\} \times \{1, \ldots, n\}$, and it is said to be {\em proper} if 1) it includes only pairs $(i,j)$ with $i > j$, and 2) it satisfies the transitivity property that if $(i,j)$ and $(j,k)$ are both elements of the board, then so too is $(i,k)$. To a proper $n$-board $B$ associate a graph $\Gamma_n(B)$ on vertex set $\{1, \ldots, n\}$ by putting an edge from $i$ to $j$ (for $i > j$) if and only if $(i,j) \not \in B$. Denote by $r_k(B)$ the number of ways of placing $k$ non-attacking rooks on $B$; that is, the number of ways of selecting a subset of $B$ of size $k$, with no two elements of the subset sharing a first coordinate, and no two sharing a second coordinate. The relevant result of Goldman et al. \cite[Theorem 2]{GoldmanJoichiWhite} is that for all $k \geq 0$, $r_k(B) = {\Gamma_n(B) \brace n-k}$. (Goldman et al. use the notation $q_{n-k}$ for ${\Gamma_n(B) \brace n-k}$.)

To interpret this result in the language of the present paper, let $w$ be a Dyck word with $n$ $x$'s, and let ${\mathcal F}_w$ be the set of (labels of) unit squares that lie above the staircase path of $w$, and inside the $[0,n] \times [0,n]$ square (note that ${\mathcal F}_w$ forms what is often called a {\em Ferrers board}). While ${\mathcal F}_w$ does not form a proper $n$-board, it is easy to check that if we let $\tilde{{\mathcal F}}_w$ be the reflection of ${\mathcal F}_w$ across the line $x=y$, then $\tilde{{\mathcal F}}_w$ does, and that the graph $\Gamma_n(\tilde{{\mathcal F}}_w)$ is isomorphic to $H_w$ (via the identity map on the labels). It is also clear that $r_k(\tilde{{\mathcal F}}_w)=r_k({\mathcal F}_w)$. Thus Goldman et al.'s result is that for all $k \geq 0$,
\begin{equation} \label{H1}
r_k({\mathcal F}_w) = {H_w \brace n-k}.
\end{equation}

The second result we need is Navon's combinatorial interpretation of $S_w(k)$ from \cite{Navon}. Just as we associated a Dyck path with a Dyck word in Section \ref{subsec-combint1}, we may associate a staircase path with an arbitrary word $w$ by starting at $(0,0)$ and, reading $w$ from left to right, taking a step in the positive $y$ direction each time an $x$ is encountered in $w$, and a step in the positive $x$ direction each time a $D$ is encountered. Let $B_w$ be the set of labels of the unit squares that lie above this staircase path and inside the box $[0,n] \times [0,m]$, where $w$ has $m$ $x$'s and $n$ $D$'s. As before, let $r_k(B_w)$ be the number of ways of placing $k$ non-attacking rooks on $B_w$. Navon's combinatorial interpretation of the numbers $S_w(k)$, as stated (and reproved) by Varvak in \cite[Theorem 3.1]{Varvak}, is that $S_w(n-k) = r_k(B_w)$. (Note that Varvak uses ``$U$'' in place of $x$).

It is clear that if $w$ is a Dyck word with $n$ $x$'s then $B_w$ from Navon's interpretation is exactly our ${\mathcal F}_w$, and so Navon's interpretation becomes in this case
\begin{equation} \label{H2}
S_w(n-k) = r_k({\mathcal F}_w)
\end{equation}
for all $k \geq 0$. Combining (\ref{H1}) and (\ref{H2}) we get Theorem \ref{thm-main_comb_int2}.

\medskip

It is also possible to give a direct proof (not using Navon's interpretation) of the identity ${H_w \brace k}={G_w \brace k}$ for all Dyck words $w$ and integers $k \geq 0$ (and so also an alternate proof of Theorem \ref{thm-main_comb_int2} via Theorem \ref{thm-main_comb_int}). We have already calculated (in Section \ref{sec-comb_int_new}) the chromatic polynomial of $G_w$ to be
$$
\chi_{G_w}(q) = \prod_{i=1}^m (q-a_i)
$$
where $w$ has $m$ $x$'s (and so also $m$ $D$'s, since we are assuming it to be a Dyck word), and $a_i$ is the height of the $i^{th}$ $x$. If we can show that $H_w$ has the same chromatic polynomial, then we get ${H_w \brace k}={G_w \brace k}$ using (\ref{lem-chromaticpoly}).

To compute $\chi_{H_w}(q)$, consider the set ${\mathcal T}_w=\{(x_1,y_1), \ldots, (x_k,y_k)\}$ (defined in Section \ref{subsec-Hintro}) of peaks of the Dyck path of $w$. As discussed in Section \ref{subsec-Hintro}, the edge set of $H_w$ is obtained by putting cliques on each of the consecutive integer segments $\{x_i, \ldots, y_i\}$, $1 \leq i \leq k$ (see Figure \ref{fig-graphHw} in Section \ref{subsec-Hintro}).
We properly $q$-color $H_w$ sequentially, starting with the clique on segment $\{x_1, \ldots, y_1\}$, which can be colored in $q(q-1)\ldots (q-(y_1-x_1))$ ways. Notice, by our alternate characterization from Section \ref{sec-comb_int_new} of the heights of the $x_i$'s in a word (if the step in the positive $y$ direction corresponding to an $x$ goes from $(a,b)$ to $(a,b+1)$, then the height of that $x$ is $b-a$), that this is the same as $\prod_{i=1}^{y_1} (q-a_i)$ (where $a_i$ is the height of the $i^{th}$ $x$).

Next we move on to the clique on segment $\{x_2, \ldots, y_2\}$. The first $y_1-x_2+1$ vertices of this clique have already been colored (since they are part of the clique on segment $\{x_1, \ldots, y_1\}$), so it remains to color the last $y_2-y_1$ vertices. The palette of colors available has size $q-(y_1-x_2+1)$, so the number of ways in which these last $y_2-y_1$ vertices of the second clique can be colored is $(q-(y_1-x_2+1))(q-(y_1-x_2+1)-1)\dots (q-(y_1-x_2+1)-(y_2-y_1-1))$; this is the same as $\prod_{i=y_1+1}^{y_2} (q-a_i)$. Continuing along the integer segment cliques in this manner, and noting that all proper $q$-colorings of $H_w$ can be achieved by this sequential scheme, we get that indeed the number of proper $q$-colorings of $H_w$ is $\prod_{i=1}^m (q-a_i)$.

\section{A new summation formula when $w=(x^sD^s)^n$} \label{sec-summ}

We now turn to Theorem \ref{thm-summ}, which deals with the case $w=(x^sD^s)^n$. From Theorem \ref{thm-main_comb_int} we know that
$$
S_w(k) = {nK_s \brace k}
$$
in this case, where $nK_s$ is the disjoint union of $n$ copies of $K_s$.

The key observation that allows us to say something sensible about ${nK_s \brace k}$ is (\ref{lem-chromaticpoly}), which we use in the following way. If $G$ consists of $a$ disjoint copies of $K_b$ together with $r$ isolated vertices, then its chromatic polynomial is
$$
\chi_G(q)=q^{r+a}(q-1)^a(q-2)^a \ldots (q-(b-1))^a.
$$
On the other hand, if $G$ consists of a fan of $a$ copies of $K_b$ ($a$ copies of $K_b$ with a single vertex that is in common to all the copies) together with $r+a-1$ isolated vertices, then the chromatic polynomial is also
$$
\chi_G(q) = q^{r+a}(q-1)^a(q-2)^a \ldots (q-(b-1))^a.
$$
So whenever we see a graph of the first kind described above, we may replace it with a graph of the second kind, without changing the values of the graph Stirling numbers.

As a warm-up we use this observation first to obtain an expression for ${nK_2 \brace k}$. The chromatic polynomial of $nK_2$ is $q^n(q-1)^{n}$, 
which is the same as the chromatic polynomial of the graph $G$ consisting of a star on $n+1$ vertices together with $n-1$ isolated vertices. So ${nK_2 \brace k}={G \brace k}$, and we can find ${G \brace k}$ easily: first, decide on a subset of size $\ell$ of the isolated vertices (perhaps empty) to be in the same block as the center of the star (which cannot be in the same block as any of the leaves of the star). The remaining vertices now form an independent set of size $2n-1-\ell$, so we use an ordinary Stirling number of the second kind to count the number of partitions of these vertices into $k-1$ classes. This leads to
$$
{nK_2 \brace k} = \sum_{\ell=0}^{n-1} \binom{n-1}{\ell}{2n-1-\ell \brace k-1}.
$$

Now we deal with the general case (with $K_s$ replacing $K_2$). Here it will be convenient to define $G_{a,b,c}$ as the graph consisting of a fan of $a$ copies of $K_b$ together with $c$ isolated vertices (for example, $G_{n,2,n-1}$ is the graph $G$ that replaced $nK_2$ in the last paragraph).

The chromatic polynomial of $nK_s$ is $\left(q(q-1)(q-2)\ldots (q-(s-1))\right)^n$, and this is the same as the chromatic polynomial of $G_{n,s,n-1}$. To partition this graph into $k$ non-empty independent sets, we first decide on a subset of size $i_1$, $0 \leq i_1 \leq n-1$, of the isolated vertices to be in the same block as the center of the fan (which cannot be in a block with any of the other vertices of the fan). The remaining vertices now form the following structure: $n$ copies of $K_{s-1}$, together with $n-1-i_1$ isolated vertices. We may replace this with $G_{n,s-1,2(n-1)-i_1}$ without changing independent set partition counts, and so
$$
{nK_s \brace k} = \sum_{i_1=0}^{n-1} \binom{n-1}{i_1}{G_{n,s-1,2(n-1)-i_1} \brace k-1}.
$$
What is ${G_{n,s-1,2(n-1)-i_1} \brace k-1}$? By first partnering the center vertex of the fan with some subset of size $i_2$ ($0 \leq i_2 \leq 2(n-1)-i_1$) of the isolated vertices, then replacing the remaining structure ($n$ disjoint copies of $K_{s-2}$, together with $2(n-1)-i_1-i_2$ isolated vertices) with $G_{n,s-2,3(n-1)-i_1-i_2}$, we get
$$
{G_{n,s-1,2(n-1)-i_1} \brace k-1} = \sum_{i_2 = 0}^{2(n-1)-i_1} \binom{2(n-1)-i_1}{i_2}{G_{n,s-2,3(n-1)-i_1-i_2} \brace k-2}
$$
and so
$$
{nK_s \brace k} = \sum_{i_1=0}^{n-1} \binom{n-1}{i_1} \sum_{i_2 = 0}^{2(n-1)-i_1} \binom{2(n-1)-i_1}{i_2}{G_{n,s-2,3(n-1)-i_1-i_2} \brace k-2}.
$$
Continuing this process we eventually reach
\begin{eqnarray*}
& {G_{n,2,(s-1)(n-1)-i_1-i_2-\ldots -i_{t-2}} \brace k-(s-2)} & \\
 & = & \\
& \sum_{i_{s-1}=0}^{(s-1)(n-1)-i_1-i_2-\ldots -i_{s-2}} \binom{(s-1)(n-1)-i_1-i_2-\ldots -i_{s-2}}{i_{s-1}} {s(n-1)+1-i_1-i_2-\ldots -i_{s-1} \brace k-(s-1)}. &
\end{eqnarray*}
(the final graph we consider on the right-hand side above is $G_{n,1,s(n-1)-i_1-i_2-\ldots -i_{s-1}}$, which is a collection of $s(n-1)+1-i_1-i_2-\ldots -i_{s-1}$ isolated vertices, so we are able to use an ordinary Stirling number to count partitions of it into $k-(s-1)$ blocks). We conclude that ${nK_s \brace k}$ equals
\begin{eqnarray}
& \sum_{i_1=0}^{n-1} \binom{n-1}{i_1}  \sum_{i_2 = 0}^{2(n-1)-i_1} \binom{2(n-1)-i_1}{i_2} \ldots & \nonumber \\
& \sum_{i_{s-2}=0}^{(s-2)(n-1)-i_1-i_2-\ldots -i_{s-3}} \binom{(s-2)(n-1)-i_1-i_2-\ldots -i_{s-3}}{i_{s-2}} & \label{first-exp} \\
& \sum_{i_{s-1}=0}^{(s-1)(n-1)-i_1-i_2-\ldots -i_{s-2}} \binom{(s-1)(n-1)-i_1-i_2-\ldots -i_{s-2}}{i_{s-1}} {s(n-1)+1-i_1-i_2-\ldots -i_{s-1} \brace k-(s-1)}. & \nonumber
\end{eqnarray}
Set $\ell=i_1+i_2+\ldots + i_{s-1}$ (so $\ell$ ranges from $0$ to $(s-1)(n-1)$). For each $\ell$ in this range the coefficient of ${s(n-1)+1-\ell \brace k-(s-1)}$ in (\ref{first-exp}) is
$$
\sum_{i_1+i_2+\ldots +i_{s-1} = \ell} \binom{n-1}{i_1} \binom{2(n-1)-i_1}{i_2} \ldots \binom{(s-1)(n-1)-i_1-i_2-\ldots -i_{s-2}}{i_{s-1}}
$$
(note that if ever we consider a $(s-1)$-tuple in the sum above that fails to satisfy one of the conditions $0 \leq i_j \leq j(n-1)-i_1-i_2-\ldots -i_{j-1}$, then the corresponding binomial coefficient will be $0$). This establishes (\ref{fnsl}), and completes the proof of Theorem \ref{thm-summ}.

To establish (\ref{fnsl2}), we must show that the right-hand side of (\ref{fnsl}) equals the right-hand side of (\ref{fnsl2}), for which we employ a combinatorial argument. Let $A_1, \ldots, A_{s-1}$ be $s-1$ disjoint sets, each of size $n-1$. An {\em $\ell$-selection} from $A_1$ through $A_{s-1}$ is a specification of sets $A_{11}$, $A_{21}$, $A_{22}$, $A_{31}$, $A_{32}$, $A_{33}$, $\ldots$ $A_{(s-1)1}$, $\ldots$, $A_{(s-1)(s-1)}$, pairwise disjoint, with $A_{11} \subseteq A_1$, $A_{21} \cup A_{22} \subseteq A_2$, $\ldots$, $A_{(s-1)1} \cup \ldots \cup A_{(s-1)(s-1)} \subseteq A_{s-1}$, and $|A_{11}|+|A_{21}| + \ldots + |A_{(s-1)(s-1)}| = \ell$.

To count the number of $\ell$-selections, we first specify a composition $\ell=i_1+i_2+\ldots+i_{s-1}$, then from each $A_k$ select a subset of size $i_k$, and then for each element of the chosen subset, decide which of $A_{k1}, \ldots A_{kk}$ the element belongs to. This gives that the number of $\ell$-selections is
$$
\sum_{i_1+\ldots+i_{s-1} = \ell} \prod_{k=1}^{s-1} k^{i_k}\binom{n-1}{i_k} =  [x^\ell]\left((1+x)(1+2x)\ldots (1+(s-1)x)\right)^{n-1}.
$$

Another way to count $\ell$-selections is to first specify a composition $\ell=i_1+i_2+\ldots+i_{s-1}$. Then, select from $A_{s-1}$ a subset of size $i_1$ to be $A_{(s-1)(s-1)}$. Next, select a subset of $(A_{s-2} \cup A_{s-1}) \setminus A_{(s-1)(s-1)}$ of size $i_2$; let its intersection with $A_{s-1}$ be $A_{(s-1)(s-2)}$, and its intersection with $A_{s-2}$ be $A_{(s-2)(s-2)}$. Next, select a subset of $(A_{s-3} \cup A_{s-2} \cup A_{s-1}) \setminus (A_{(s-1)(s-1)} \cup A_{(s-1)(s-2)} \cup A_{(s-2)(s-2)})$ of size $i_3$; let its intersection with $A_{s-1}$ be $A_{(s-1)(s-3)}$, its intersection with $A_{s-2}$ be $A_{(s-2)(s-3)}$ and its intersection with $A_{s-3}$ be $A_{(s-3)(s-3)}$. Continue in this manner, until finally we are selecting a subset of size $i_1$ of the as-yet-unselected elements of $A_1 \cup \ldots \cup A_{s-1}$; for each $k$, $1 \leq k \leq s-1$, let its intersection with $A_k$ be $A_{k1}$. This gives that the number of $\ell$-selections is
$$
\sum_{i_1+i_2+\ldots +i_{s-1} = \ell} \binom{n-1}{i_1} \binom{2(n-1)-i_1}{i_2} \ldots \binom{(s-1)(n-1)-i_1-i_2-\ldots -i_{s-2}}{i_{s-1}}.
$$
This completes the verification of (\ref{fnsl2}).

\section{Bijections involving increasing forests} \label{sec-forests}

Recall that for the word $w = (x^{r} D)^n$, we set $G_w = G(n,r)$ and $H_w = H(n,r)$; in this case $G(n,r)$ is the threshold graph obtained from the empty graph by iterating $n$ times the operation of adding an isolated vertex followed by $r-1$ dominating vertices, and $H(n,r)$ is the indifference graph on vertex set $\{1,\ldots,rn\}$ with edges covered by cliques on vertices $i$ through $ir$, for $1 \leq i \leq n$. In this section we will exhibit bijections between the set of partitions $G(n,r)$ (and $H(n,r)$) into $k+(r-1)n$ independent sets and the set of \emph{decreasing} $r$-ary forests with $n$ vertices and $k$ components (where decreasing means that the labels go in decreasing order when read along any path starting from the root vertex of a tree). Decreasing forests turn out to be notationally a little easier to deal with than increasing forests, but of course via the order-reversing permutation $x \mapsto n+1-x$ there is a perfect correspondence between the two.

\subsection{Bijection for $G(n,r)$} \label{sec-bij1}

Here we provide a bijective proof that
\[
{G(n,r) \brace k + (r-1)n}  =  |F'(r,n,k)|,
\]
where $F'(r,n,k)$ is the set of decreasing $r$-ary forests with $n$ vertices and $k$ components.
It will be useful to exhibit the threshold graph $G(n,r)$ concisely.  To that end, we will use `$+$' to indicate the addition of a dominating vertex and `$\bullet$' to indicate the addition of an isolated vertex, presenting the vertices from right to left (see Figure \ref{fig-G(n,r)}).  In particular, $+$'s are adjacent to every vertex to their right and also every $+$ to their left, while $\bullet$'s are only adjacent to the $+$'s to their left.
In the concise representation of $G(n,r)$, label the vertices in increasing order from left to right; we will often abuse notation and identify a vertex with its label.

\begin{figure}[ht]
\[
\begin{array}{cccccccccccccccccccc}
1&2&3&4&5&6&7&8&9&10&11&12&13&14&15&16&17&18&19&20\\
+&+&+&\bullet&+&+&+&\bullet&+&+&+&\bullet&+&+&+&\bullet&+&+&+&\bullet\\
\end{array}
\]
\caption{The concise notation for the graph $G(5,4)$, with labels.}
\label{fig-G(n,r)}
\end{figure}

We also note in passing that $G(n,r)$ may be constructed from the half-graph (on $2n$ vertices $x_1, \ldots, x_n,y_1, \ldots, y_n$ with $x_i$ joined to $y_j$ when $j \leq i$), by adding edges of the form $y_iy_j$ for $i \neq j$ and blowing up each $y_i$ to a copy of $K_{r-1}$ (see Figure \ref{fig-Gw}).

\begin{figure}[ht]
\begin{center}
\begin{tikzpicture}[scale=2.5]

    \node (v1) at (1,2) [circle,draw,fill] {};
	\node (v2) at (2,2) [circle,draw,fill] {};
	\node (v3) at (3,2) [circle,draw,fill] {};
	\node (v4) at (4,2) [circle,draw,fill] {};
	\node (v5) at (5,2) [circle,draw,fill] {};
	\node (w1) at (1,1) [circle,draw] {$K_{r-1}$};
	\node (w2) at (2,1) [circle,draw] {$K_{r-1}$};
	\node (w3) at (3,1) [circle,draw] {$K_{r-1}$};
	\node (w4) at (4,1) [circle,draw] {$K_{r-1}$};
	\node (w5) at (5,1) [circle,draw] {$K_{r-1}$};
	
	\draw (.66,.66) -- (.66,1.34);
	\draw (.66,1.34) -- (5.34,1.34);
	\draw (5.34,1.34) -- (5.34,.66);
	\draw (5.34,.66) -- (.66,.66);
	\node at (5.8,1) {$\Bigg\}$ complete};

	\foreach \from/\to in {v5/w1,v5/w2,v5/w3,v5/w4,v5/w5,v4/w1,v4/w2,v4/w3,v4/w4,v3/w1,v3/w2,v3/w3,v2/w1,v2/w2,v1/w1}
	\draw (\from) -- (\to);

\end{tikzpicture}
\caption{The graph $G(5,r)$. An edge from a vertex on top to a $K_{r-1}$ below indicates that the vertex on top is joined to all vertices of $K_{r-1}$.}

\label{fig-Gw}
\end{center}
\end{figure}
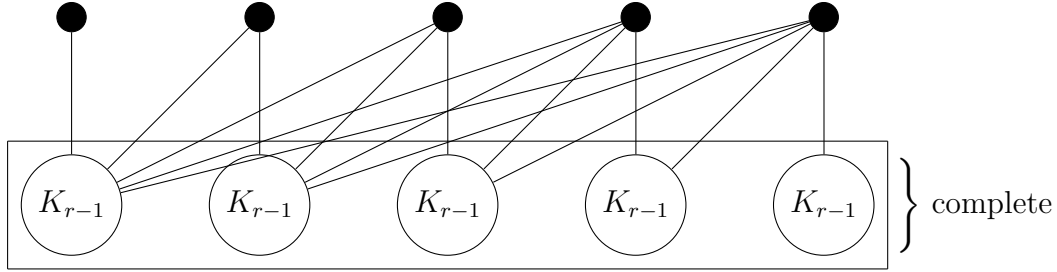

Let $\mathcal{P}$ be the set of partitions of $G(n,r)$ into $k + (r-1)n$ independent sets.  We begin by describing a map from $\mathcal{P}$ to $F'(r,n,k)$.  Fix $P \in \mathcal{P}$, which corresponds to a partition of $\{1,\ldots,rn\}$ via the labels on the vertices. Note that the vertices in any subset of $V(G(n,r))$, and so in particular the independent sets that make up $P$, are totally ordered by the labeling of the vertices.
We will iteratively build a decreasing $r$-ary forest from $P$.

Let $M$ denote the set of vertices of the form $ar$ for $1\leq a \leq n$ (these are the $\bullet$'s from the concise representation of $G(n,r)$). The vertices not in $M$, together with vertex $rn$, form a clique in $G(n,r)$ and so must be in distinct independent sets in $P$. This means that there are $k$ independent sets in $P$ that contain only elements from $M$. For each of these $k$ independent sets we place a root in the forest with label $j$, where $jr$ is the largest label in the corresponding independent set.

We use the location of the remaining vertices in $M$ to iteratively construct the rest of the forest, and we do so by considering these $n-k$  vertices in decreasing order.  Suppose that $xr$ is the largest vertex in $M$ that has not been considered, meaning in particular that $xr$ is not the largest vertex in an independent set consisting of elements of $M$.
Among the elements of $xr$'s independent set with label larger than $xr$, let $z$ be the element with the smallest label.
(Notice that such a $z$ must exist: $xr$ must be in an independent set with either a vertex from $M$ with a larger label, or some vertex outside of $M$, which would necessarily have a larger label than $xr$ since $xr$ is adjacent to all vertices outside of $M$ with smaller labels.)

There are unique positive integers $1 \leq a \leq n$ and $1 \leq p \leq r$ such that $z = (a-1)r + p$. We place a vertex labeled $x$ in the forest in the slot reserved for the $p^{th}$ child of the vertex labeled $a$. (Notice that $a$ has indeed been placed in the forest already, since $xr < z=(a-1)r+p \leq ar$ and we consider the elements of $M$ in decreasing order.)

Since distinct $z$'s are associated to distinct $xr$'s, and the pair $(a,p)$ is uniquely determined by $z$, we will never be forced to place distinct vertices in the same location as this iterative process goes on, and so the final result will indeed be an $n$-vertex, $k$-component, $r$-ary forest; and
the fact that $x < a$ and $x$ is placed below $a$ in the forest means that the labeling is decreasing.

The whole process is reversible. We can see this by considering the vertices of the forest in order of increasing labels (from $1$ to $n$). For each label $x$ we can use $p$ and $a$ (which are known from the location of $x$ in the forest) to obtain $z$, which is the next largest vertex in the independent set containing $xr$.  In this way, we recover the partition of $G(n,r)$ into $k + (r-1)n$ independent sets that led to the forest. This shows that our map is indeed a bijection.

\subsection{Bijection for $H(n,r)$} \label{sec-bij2}

We now provide a bijective proof that
\[
{H(n,r) \brace k + (r-1)n}  =  |F'(r,n,k)|,
\]
where $H(n,r)$ is the graph on vertex set $\{1,\ldots,rn\}$ with edges covered by cliques on vertices $i$ through $ir$, for $i=1,\ldots,n$ (see Figure \ref{fig-Hw}).

\begin{figure}[ht!]
\begin{center}
\begin{tikzpicture}[scale=1]

	\node (v1) at (1,1) [circle,draw] {};
	\node at (1,-.25) {$1$};
	\node (v2) at (2,1) [circle,draw] {};
	\node at (2,-.25) {$2$};
	\node (v3) at (3,1) [circle,draw] {};
	\node at (3,-.25) {$3$};
	\node (v4) at (4,1) [circle,draw] {};
	\node at (4,-.25) {$4$};
	\node (v5) at (5,1) [circle,draw] {};
	\node at (5,-.25) {$5$};
	\node (v6) at (6,1) [circle,draw] {};
	\node (v7) at (7,1) [circle,draw] {};
	\node (v8) at (8,1) [circle,draw] {};
	\node (v9) at (9,1) [circle,draw] {};
	\node (v10) at (10,1) [circle,draw] {};
	\node (v11) at (11,1) [circle,draw] {};
	\node (v12) at (12,1) [circle,draw] {};
	\node (v13) at (13,1) [circle,draw] {};
	\node (v14) at (14,1) [circle,draw] {};
	\node (v15) at (15,1) [circle,draw] {};
	
	\draw (.5,.5) -- (.5,1.5);
	\draw (.5,1.5) -- (3.45,1.5);
	\draw (3.45,1.5) -- (3.45,.5);
	\draw (3.45,.5) -- (.5,.5);
	
	\draw [red] (1.5,.4) -- (1.5,1.6);
	\draw [red] (1.5,1.6) -- (6.5,1.6);
	\draw [red] (6.5,1.6) -- (6.5,.4);
	\draw [red] (6.5,.4) -- (1.5,.4);

	\draw [blue] (2.5,.3) -- (2.5,1.7);
	\draw [blue] (2.5,1.7) -- (9.5,1.7);
	\draw [blue] (9.5,1.7) -- (9.5,.3);
	\draw [blue] (9.5,.3) -- (2.5,.3);

	\draw [green] (3.55,.2) -- (3.55,1.8);
	\draw [green] (3.55,1.8) -- (12.5,1.8);
	\draw [green] (12.5,1.8) -- (12.5,.2);
	\draw [green] (12.5,.2) -- (3.55,.2);

	\draw [orange] (4.5,.1) -- (4.5,1.9);
	\draw [orange] (4.5,1.9) -- (15.5,1.9);
	\draw [orange] (15.5,1.9) -- (15.5,.1);
	\draw [orange] (15.5,.1) -- (4.5,.1);

\end{tikzpicture}
\caption{The graph $H(5,3)$.}
\label{fig-Hw}
\end{center}
\end{figure}
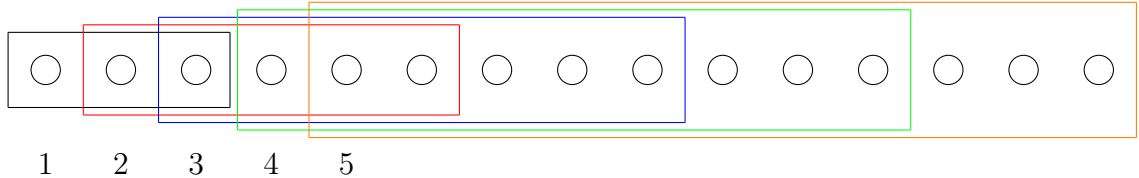

Let ${\mathcal P}'$ be the set of partitions of $H(n,r)$ into $k+(r-1)n$ independent sets.
As before, we start by describing a map from ${\mathcal P}'$ to $F'(r,n,k)$. Fix $P' \in {\mathcal P}'$. This corresponds to a partition of $\{1,\ldots,rn\}$ via the labels on the vertices (we will again abuse notation and refer to the vertices by their labels).  Place the vertices of each independent set into an increasing ordered list.  We will iteratively build a decreasing $r$-ary forest from $P'$.

If $i \leq n$ is the largest label in an independent set, put a root in the forest with label $i$.  This gives $k$ roots for the $k+(r-1)n$ independent sets. We will consider the remaining vertices in $\{1,\ldots,n\}$ in decreasing order.

Suppose that $x$ is the largest vertex that has not been placed in the forest, and suppose that $x$ immediately precedes $y$ in the list for some independent set (so $x < n$, and $y > xr$).  Then find the unique positive integers $2 \leq a \leq n$ and $1 \leq p \leq r$ satisfying $y = (a-1)r + p$.  Place a vertex labeled $x$ in the $p^{th}$ possible position (again scanning from left to right, say) below the vertex labeled $a$ in the forest.  Using the increasing order on the independent set assures that we do not try to place two labels in the same slot beneath a vertex in the forest, since each $y$ will only be considered once and each $y$ is in bijective correspondence with a pair $(a,p)$.

Now $x$ and $y$ are in the same independent set (and no vertex with a label between that of $x$ and $y$ is in the independent set), so $rx < y = (a-1)r + p \leq ra$. This implies that $x < a$.  Therefore this iterative procedure produces a decreasing $r$-ary forest with $k$-components. As with the map described in Section \ref{sec-bij1}, this process is reversible. We see this by considering the vertices of a forest in order of increasing labels. For each $x$ we can use $p$ and $a$ (which are known from the location of $x$ in the forest) to obtain $y$, which is the next label in the independent set containing $x$.  In this way, we recover the partition of $H(n,r)$ into $k + (r-1)n$ independent sets that led to the forest, showing that our map is a bijection.

\section{$q$-analogs} \label{sec-q}

The proof of Theorem \ref{thm-q}, like that of Theorem \ref{thm-main_comb_int}, is based on two claims.
\begin{claim} \label{claim-domq}
Let $w=xw'D$ be an irreducible Dyck word in the $q$-deformed Weyl algebra, with $w'$ non-empty. For all $k \geq 0$ we have
$$
S^q_w(k) = S^q_{w'}(k-1),
$$
where $S^q_{w'}(-1)$ is interpreted as $0$.
\end{claim}

\begin{proof}
We have $w=\sum_{k \geq 0} S^q_w(k) x^kD^k$, but also
\begin{eqnarray*}
w & = & xw'D \\
& = &  x\left(\sum_{k \geq 0} S^q_{w'}(k-1) x^{k-1}D^{k-1}\right)D \\
& = & \sum_{k \geq 0} S^q_{w'}(k-1) x^kD^k.
\end{eqnarray*}
The claim now follows by the uniqueness of the representation in (\ref{eq-weyl_general_q}).
\end{proof}

For the second claim, we define the $q$-factorial (for integer $a$) via
$$
[a]_q! = \left\{
\begin{array}{ll}
[a]_q [a-1]_q [a-2]_q \ldots [2]_q [1]_q & \mbox{if $a \geq 1$}, \\
1 & \mbox{if $a=0$, and} \\
0 & \mbox{if $a < 0$},
\end{array}
\right.
$$
where recall that for positive integer $m$, $[m]_q = 1+q+\ldots +q^{m-1}$, and define the $q$-binomial coefficient (for integers $a$ and $b$) via
$$
\binom{a}{b}_q = \frac{[a]_q!}{[b]_q![a-b]_q!}.
$$
\begin{claim} \label{claim-unionq}
Let $w_1, w_2$ be Dyck words in the $q$-deformed Weyl algebra.
For all $k \geq 0$ we have
$$
S^q_{w_1w_2}(k) = \sum_{r, s \geq 0} S^q_{w_1}(r)S^q_{w_2}(s) q^{(r-k)(s-k)} \binom{r}{r+s-k}_q \binom{s}{r+s-k}_q [r+s-k]_q!
$$
where $w_1w_2$ is the concatenation of $w_1$ and $w_2$.
\end{claim}

\begin{proof}
We begin by deriving the following analog of (\ref{Leibnitz_q=1}) in the $q$-deformed Weyl algebra:
\begin{equation} \label{Leibnitz_q}
x^r D^r x^s D^s = \sum_{j \geq 0} q^{(r-j)(s-j)}\frac{\left([r]_q\right)^{\underline j} \left([s]_q\right)^{\underline j}}{[j]_q!} x^{r+s-j}D^{r+s-j}.
\end{equation}
We prove this by induction on $r+s$. If $rs=0$ then the result is trivial. So, consider a pair $(r,s)$ with $r+s\geq 2$ and $rs>0$. Replace the right-most $Dx$ in $x^r D^r x^s D^s$ with $qxD+1$ to obtain a sum of two terms, one of which still has $r+s$ $x$'s and $r+s$ $D$'s. In this term, again replace the right-most $Dx$ with $qxD+1$. Iterating this operation $r$ times we obtain
$$
x^r D^r x^s D^s = x\left(q^r x^r D^r x^{s-1} D^{s-1} + [r]_q x^{r-1} D^{r-1} x^{s-1} D^{s-1}\right) D.
$$
By induction, the right-hand side above is the sum of
$$
\sum_{j \geq 0} q^{(r-j)(s-1-j)+r} \frac{([r]_q)^{\underline j}([s-1]_q)^{\underline j}}{[j]_q!} x^{r+s-j}D^{r+s-j}
$$
and
$$
\sum_{j \geq 0} [r]_q q^{(r-1-j)(s-1-j)} \frac{([r-1]_q)^{\underline j}([s-1]_q)^{\underline j}}{[j]_q!} x^{r+s-j-1}D^{r+s-j-1}.
$$
The coefficient of $x^{r+s}D^{r+s}$ in the normal order of $x^r D^r x^s D^s$ comes from the $j=0$ term of the first sum above, and is $q^{rs}$, as required. For $j>0$ the coefficient of $x^{r+s-j}D^{r+s-j}$ is
$$
q^{(r-j)(s-1-j)+r} \frac{([r]_q)^{\underline j}([s-1]_q)^{\underline j}}{[j]_q!} + [r]_q q^{(r-j)(s-j)} \frac{([r-1]_q)^{\underline{j-1}}([s-1]_q)^{\underline{j-1}}}{[j-1]_q!}
$$
which, after a little algebra, simplifies to
$$
q^{(r-j)(s-j)}\frac{\left([r]_q\right)^{\underline j} \left([s]_q\right)^{\underline j}}{[j]_q!},
$$
completing the induction. Armed with (\ref{Leibnitz_q}), we complete the proof of the claim exactly as in the derivation of (\ref{induction}).
\end{proof}

We also need the following lemma. Let $K_{r,s}$ be the complete bipartite graph with partition classes $X=\{x_1, \ldots, x_r\}$ and $Y=\{y_1, \ldots, y_s\}$. Let ${\mathcal M}_k(K_{r,s})$ be the set of matchings of size $k$ in $K_{r,s}$ (selections of $k$ edges from $K_{r,s}$, no two sharing an endvertex). Encode a matching $M \in {\mathcal M}_k(K_{r,s})$ by an $r$ by $s$ matrix (which we also call $M$) whose $ij$ entry is $1$ if $\{x_i,y_j\}$ is in the matching, and $0$ otherwise, and let $f(M)$ be as in Definition \ref{def-f(M)}; that is, $f(M)$ is the number of unmarked $0$'s in $M$ after marking all $0$'s that occur below or to the right of a $1$.
\begin{lemma} \label{lem-q-matching}
With the notation as above, for all $k \geq 0$ we have
\begin{equation} \label{q-sum}
\sum_{M \in {\mathcal M}_k(K_{r,s})} q^{f(M)} = q^{(r-k)(s-k)} \binom{r}{k}_q \binom{s}{k}_q [k]_q!.
\end{equation}
\end{lemma}
Notice that when $q=1$, (\ref{q-sum}) reduces to $|{\mathcal M}_k(K_{r,s})| = \binom{r}{k} \binom{s}{k} (k)!$,
which is evident: to specify a matching of size $k$ one first chooses subsets of $X$ and $Y$, each of size $k$ (there are $\binom{r}{k} \binom{s}{k}$ ways to do this) and then chooses how to match them up (there are $k!$ ways to do this).

\begin{proof} (Lemma \ref{lem-q-matching})
We use the following well-known interpretations of $\binom{r}{k}_q$ and $[k]_q!$ (see for example \cite[Theorem 3.6]{Andrews} for a statement that encompasses both). First, if ${\mathbb Q}^{(k)}_r$ is the set of $0$-$1$ strings of length $r$ that have exactly $k$ $1$'s, and for each such string $\sigma$,
${\rm zeros}(\sigma)$ is the sum, over all the $1$'s in $\sigma$, of the number of $0$'s to the left of that $1$, then
$$
\binom{r}{k}_q = \sum_{\sigma \in {\mathbb Q}^{(k)}_r} q^{{\rm zeros}(\sigma)}.
$$

Second, if ${\mathbb P}_k$ is the set of permutations of $\{1, \ldots, k\}$ (written in one-line notation), and for each such permutation $\pi$, ${\rm inv}(\pi)$ counts the number of inversions in $\pi$ (the number of pairs $(i,j)$ with $i < j$ such that $j$ appears before $i$ in $\pi$), then
$$
[k]_q! = \sum_{\pi \in {\mathbb P}_k} q^{{\rm inv}(\pi)}.
$$

From these interpretations we see that $q^{(r-k)(s-k)} \binom{r}{k}_q \binom{s}{k}_q [k]_q!$ expands out to a sum of monomials of the form $q^{(r-k)(s-k)+{\rm zeros}(\sigma)+{\rm zeros}(\tau)+ {\rm inv}(\pi)}$ where $\sigma \in {\mathbb Q}^{(k)}_r$, $\tau \in {\mathbb Q}^{(k)}_s$ and $\pi \in {\mathbb P}_k$, with exactly one such monomial for each triple $(\sigma,\tau,\pi)$. We will prove (\ref{q-sum}) by exhibiting a bijection $\varphi$ from ${\mathcal M}_k(K_{r,s})$ to ${\mathbb Q}^{(k)}_r \times {\mathbb Q}^{(k)}_s \times {\mathbb P}_k$ with the property that if $\varphi(M)=(\sigma,\tau,\pi)$ then $f(M) = (r-k)(s-k)+{\rm zeros}(\sigma)+{\rm zeros}(\tau)+ {\rm inv}(\pi)$.

The map $\varphi$ is defined as follows. For each $M \in {\mathcal M}_k(K_{r,s})$, we let $\sigma$ be the $0$-$1$ string of length $r$ that has a $1$ in the $i^{th}$ position exactly when row $i$ of $M$ has a $1$ (note that there are exactly $k$ $1$'s in $\sigma$) and we let $\tau$ be the $0$-$1$ string of length $s$ that has a $1$ in the $j^{th}$ position exactly when column $j$ of $M$ has a $1$ (again, this is a vector with exactly $k$ $1$'s). Finally, to define $\pi$, we let $M'$ be the (unique) $k$ by $k$ submatrix of $M$ that includes all $k$ of the $1$'s. Relabeling the rows and columns of $M'$ by $1$ through $k$ in the natural way, let the locations of the $1$'s in $M'$ be $(i_1,1), (i_2,2), \ldots, (i_k,k)$. The permutation $\pi$ is then $i_1\ldots i_k$.

For example, suppose that $r=7$, $s=6$, $k=4$ and the matching $M$ consists of the edges $\{x_2,y_3\}$, $\{x_4,y_4\}$, $\{x_5,y_1\}$ and $\{x_6,y_6\}$ (see Figure \ref{fig-q analog}; the marked entries in $M$ are identified with a ``$\star$''). We have $f(M)=19$, $\sigma=0101110$, $\tau=101101$, $\pi=(3,1,2,4)$, $(r-k)(s-k)=6$, ${\rm zeros}(\sigma)=7$, ${\rm zeros}(\tau)=4$ and ${\rm inv}(\pi)=2$.

\begin{figure}[ht]
\[
\left( \begin{array}{cccccc}
0&0&0&0&0&0\\
0&0&1&\star&\star&\star\\
0&0&\star&0&0&0\\
0&0&\star&1&\star&\star\\
1&\star&\star&\star&\star&\star\\
\star&0&\star&\star&0&1\\
\star&0&\star&\star&0&\star
\end{array}
\right)
\]
\caption{An example matrix $M$, with the marked $0$'s in $M$ identified with a ``$\star$.''}
\label{fig-q analog}
\end{figure}

We first note that $\varphi$ is a bijection from ${\mathcal M}_k(K_{r,s})$ to ${\mathbb Q}^{(k)}_r \times {\mathbb Q}^{(k)}_s \times {\mathbb P}_k$, since the location of $M'$ can be reconstructed from $\sigma$ and $\tau$, and the exact location of the $1$'s in $M$ is then determined by $\pi$.

To see $f(M)=(r-k)(s-k)+{\rm zeros}(\sigma)+{\rm zeros}(\tau)+ {\rm inv}(\pi)$, first note that there are exactly $(r-k)(s-k)$ entries $ij$ in $M$, all unmarked $0$'s, with no $1$ in row $i$ and no $1$ in column $j$ (these entries are labeled ``$a$'' in the matrix in Figure \ref{fig-q analog partition}). Call this set of entries $M_0$.

\begin{figure}[ht]
\[
\left( \begin{array}{cccccc}
b&a&b&b&a&b\\
d&c&\bullet&\bullet&\bullet&\bullet\\
b&a&\bullet&b&a&b\\
d&c&\bullet&\bullet&\bullet&\bullet\\
\bullet&\bullet&\bullet&\bullet&\bullet&\bullet\\
\bullet&c&\bullet&\bullet&c&\bullet\\
\bullet&a&\bullet&\bullet&a&\bullet
\end{array}
\right)
\]
\caption{The unmarked $0$'s of $M$, partitioned into $M_0$ ($a$'s), $M_{\sigma}$ ($b$'s), $M_{\tau}$ ($c$'s), and $M'$ ($d$'s).}
\label{fig-q analog partition}
\end{figure}

Next, consider all of the entries in $M$ that lie above (and in the same column as) a $1$, but are not in $M'$ (these entries are labeled ``$b$'' in Figure \ref{fig-q analog partition}). They are all unmarked $0$'s (being above a $1$, each such entry is not below a $1$, and being outside of $M'$, it is not to the right of a $1$), and there are exactly ${\rm zeros}(\sigma)$ of them. Call this set of entries $M_\sigma$; note that $M_0$ and $M_\sigma$ are disjoint.

Next, consider all of the entries in $M$ that lie to the left (and in the same row as) a $1$, but are not in $M'$ (these entries are labeled ``$c$'' in Figure \ref{fig-q analog partition}). They are all unmarked $0$'s (being to the left of a $1$, each such entry is not to the right of a $1$, and being outside of $M'$, it is not below a $1$), and there are exactly ${\rm zeros}(\tau)$ of them. Call this set of entries $M_\tau$; note that $M_0$ and $M_\tau$ are evidently disjoint, and that also $M_\sigma$ and $M_\tau$ are disjoint --- an entry in the intersection of $M_\sigma$ and $M_\tau$ would have to be both above a $1$ and to the left of a $1$, and so be in $M'$.

Note that $M_0 \cup M_\sigma \cup M_\tau$ is exactly the set of unmarked $0$'s outside of $M'$, so we are done if we can show that the number of unmarked $0$'s in $M'$ (entries labeled ``$d$'' in Figure \ref{fig-q analog partition}) is ${\rm inv}(\pi)$. Consider such an unmarked $0$, at position $(i,j)$. It must be located above some $1$, at position $(i',j)$, say, and to the left of some other $1$, at position $(i,j')$, say. This gives rise to a pair of $1$'s at positions $(i',j)$ and $(i,j')$, with $j<j'$ and $i<i'$, which is an instance of an inversion in $\pi$; and conversely, any inversion in $\pi$ is easily seen to correspond to an unmarked $0$ in $M'$.
\end{proof}

To prove Theorem \ref{thm-q}, we proceed by induction on the length of $w$. If $w$ has length $2$ then $w=xD$, $G_w=K_1$, and the result is trivial.

If $w$ is irreducible and of length greater than $2$, then $w=xw'D$ for some Dyck word $w$. By induction, we have that for each $k$,
$$
\sum_{P' \in {\mathcal P}(w',k-1)} q^{{\rm wt}(P')} = S_{w'}^q(k-1).
$$
There is a one-to-one correspondence between ${\mathcal P}(w',k-1)$ and ${\mathcal P}(w,k)$ (as discussed in Section \ref{subsec-q}), and for each $P' \in {\mathcal P}(w',k-1)$ with corresponding partition $P \in {\mathcal P}(w,k)$ we have, by definition, ${\rm wt}(P)={\rm wt}(P')$. So, using Claim \ref{claim-domq},
\begin{eqnarray*}
\sum_{P \in {\mathcal P}(w',k-1)} q^{{\rm wt}(P)} & = & \sum_{P' \in {\mathcal P'}(w',k-1)} q^{{\rm wt}(P')} \\
& = & S_{w'}^q(k-1) \\
& = & S_w^q(k)
\end{eqnarray*}
as required.

There remains the case where $w=w_1w_2 \ldots w_\ell$ is reducible (with each $w_i$ irreducible). We assume $\ell=2$; this will not lose us any generality, but will allow us to simplify notation by writing ``$w_2$'' for ``$w_2 \ldots w_\ell$'' everywhere. Combining Claim \ref{claim-unionq}, the induction hypothesis, and Lemma \ref{lem-q-matching}, what we need to show is that, for each $k$,
$$
\sum_{P \in {\mathcal P}(w_1w_2,k)} q^{{\rm wt}(P)} = \sum_{r, s \geq 0,~ P_1 \in {\mathcal P}(w_1,r), ~P_2 \in {\mathcal P}(w_2,s), ~M \in {\mathcal M}_{r+s-k}(K_{r,s})} q^{{\rm wt}(P_1)+{\rm wt}(P_2)+f(M)}
$$
(note that $(r-(r+s-k))(s-(r+s-k))=(r-k)(s-k)$). But this is immediate: we obtain all partitions of $G_{w_1w_2}$ into $k$ non-empty independent sets (and nothing more) by selecting integers $r$ and $s$, selecting a partition $P_1$ of $G_{w_1}$ into $r$ non-empty independent sets and a partition $P_2$ of $G_{w_2}$ into $s$ non-empty independent sets, and selecting a matching $M$ of size $r+s-k$ from $K_{r,s}$ which determines how $P_1$ and $P_2$ are merged; and by definition the weight ${\rm wt}(P)$ of such a matching produced by selecting a particular $r, s, P_1, P_2$ and $M$ is equal to ${\rm wt}(P_1)+{\rm wt}(P_2)+f(M)$.

\end{document}